\documentclass[11pt]{article}

\usepackage{amsmath}
\usepackage{amssymb}
\usepackage{amsmath}
\usepackage{amsthm}
\usepackage{color}
\usepackage{tikz}
\usepackage{url}

\usepackage{times}

\theoremstyle{definition} \newtheorem{definition}{Definition}[section]
\theoremstyle{plain} \newtheorem{theorem}[definition]{Theorem}
\theoremstyle{plain} \newtheorem{proposition}[definition]{Proposition}
\theoremstyle{plain} \newtheorem{conjecture}[definition]{Conjecture}
\theoremstyle{plain} \newtheorem{corollary}[definition]{Corollary}
\theoremstyle{plain} \newtheorem{lemma}[definition]{Lemma}
\theoremstyle{plain} \newtheorem{observation}[definition]{Observation}
\theoremstyle{plain} 
\theoremstyle{definition} 
\theoremstyle{plain} \newtheorem{fact}[definition]{Fact}
\theoremstyle{definition} \newtheorem{question}{Question}
\theoremstyle{remark} 

\newcommand{\eps}{\ensuremath{\varepsilon}}

\newcommand{\mcC}{\ensuremath{\mathcal{C}}}
\newcommand{\mcG}{\ensuremath{\mathcal{G}}}
\newcommand{\mcP}{\ensuremath{\mathcal{P}}}
\newcommand{\mfX}{\ensuremath{\mathfrak{X}}}
\newcommand{\mfY}{\ensuremath{\mathfrak{Y}}}
\newcommand{\bbE}{\ensuremath{\mathbb{E}}}
\newcommand{\bbN}{\ensuremath{\mathbb{N}}}
\newcommand{\bbR}{\ensuremath{\mathbb{R}}}
\newcommand{\wtO}{\ensuremath{\widetilde{O}}}

\DeclareMathOperator{\Sym}{Sym}

\DeclareMathOperator{\Aut}{Aut}

\DeclareMathOperator{\Iso}{Iso}
\DeclareMathOperator{\dist}{dist}

\title{Structure and automorphisms of primitive coherent
configurations\thanks{An extended abstract of this paper appeared in the
Proceedings of the 47th ACM Symposium on Theory of Computing (STOC'15) under
the title \emph{Faster canonical forms for primitive coherent
configurations}.}}

\author{
Xiaorui Sun\thanks{{\tt xiaoruisun@cs.columbia.edu.} This work was partially supported by a grant from the Simons Foundation (\#320173 to Xiaorui Sun).}
\\Columbia University
\and
John Wilmes\thanks{{\tt wilmesj@math.uchicago.edu.} Research supported in part by NSF grant DGE-1144082.}
\\University of Chicago
}

\date{}

\begin{document}

\maketitle


\begin{abstract}
    Coherent configurations (CCs) are highly regular colorings of the 
    set of ordered pairs of a ``vertex set''; each color represents a 
    ``constituent digraph.''  CCs arise in the study of permutation groups,
    combinatorial structures such as partially balanced designs, and the
    analysis of algorithms; their history goes back to Schur in the 1930s.  A
    CC is primitive (PCC) if all its constituent digraphs are connected.   

    We address the problem of classifying PCCs with large automorphism
    groups.  This project was started in Babai's 1981 paper in which
    he showed that only the trivial PCC admits more than 
    $\exp(\wtO(n^{1/2}))$ automorphisms.  (Here, $n$ is the number of
    vertices and the $\wtO$ hides polylogarithmic factors.) 

    In the present paper we classify all PCCs with more than
    $\exp(\wtO(n^{1/3}))$ automorphisms, making the first progress on
    Babai's conjectured classification of all PCCs with more than
    $\exp(n^{\epsilon})$ automorphisms.

    A corollary to Babai's 1981 result solved a then 100-year-old problem
    on primitive but not doubly transitive permutation groups, giving an 
    $\exp(\wtO(n^{1/2}))$ bound on their order.  In a similar vein, our 
    result implies an $\exp(\wtO(n^{1/3}))$ upper bound on the order of 
    such groups, with known exceptions.  This improvement of Babai's result 
    was previously known only through the Classification of Finite Simple 
    Groups (Cameron, 1981), while our proof, like Babai's, is elementary 
    and almost purely combinatorial.  


    Our analysis relies on a new combinatorial structure theory we develop
    for PCCs. In particular, we demonstrate the presence of 
    ``asymptotically uniform clique geometries'' on PCCs in a
    certain range of the parameters.
\end{abstract}

\section{Introduction}\label{sec:intro}

Let $V$ be a finite set; we call the elements of $V$ ``vertices.''
A \emph{configuration} of rank $r$ is a coloring $c : V\times V \to
\{0,\ldots,r-1\}$ such that (i) $c(u,u) \ne
c(v,w)$ for any $v \ne w$, and (ii) for all $i < r$ there is $i^* < r$ such
that $c(u,v) = i$ iff $c(v,u)=i^*$. The configuration is \emph{coherent} (CC)
if (iii) for all $i,j,k < r$ there is a \emph{structure constant} $p^i_{jk}$
such that if $c(u,v) = i$, there are exactly $p^i_{jk}$ vertices $w$ such that
$c(u,w)=j$ and $c(w,v)=k$. The diagonal colors $c(u,u)$ are the \emph{vertex
colors}, and the off-diagonal colors are the \emph{edge colors}. A CC is
\emph{homogeneous} (HCC) if (iv) there is only one vertex color. We denote by
$R_i$ the set of ordered pairs $(u,v)$ of color $c(u,v)=i$.  The directed graph $X_i =
(V,R_i)$ is the \emph{color-$i$ constituent digraph}. An HCC is
\emph{primitive} (PCC) if each constituent digraph is strongly connected. An
\emph{association scheme} is an HCC for which $i = i^*$ for all colors $i$
(so the constituent graphs $X_i$ are viewed as undirected).

The term ``coherent configuration'' was coined by Donald Higman in 
1969~\cite{higman-cc}, but the essential objects are older. In the case
corresponding to a permutation group, CCs already effectively appeared
in Schur's 1933 paper~\cite{schur-permutation}. This group-theoretic
perspective on CCs was developed further by
Wielandt~\cite{wielandt-permutation}.

CCs appeared for the first time from a combinatorial perspective in a 1952
paper by Bose and Shimamoto~\cite{bs-as}.  They, along with many of the
subsequent authors, consider the case of an association scheme, which is
essential for understanding partially balanced incomplete block designs, of
interest to statisticians and to combinatorial design theorists. The
generalization of an association scheme to an HCC was considered by Nair in
1964~\cite{nair-designs}. The algebra associated with a CC, which already
appeared in Schur's paper, was rediscovered in 1959 in the context of
association schemes by Bose and Mesner~\cite{bm-algebra}.

Weisfeiler and Leman~\cite{wl-wl} and Higman~\cite{higman-cc} independently
defined CCs in their full generality, including the associated algebra
(called ``cellular algebras'' by Weisfeiler and Leman), in the
late 1960s. For Higman, CCs were a generalization of permutation groups,
whereas Weisfeiler and Leman were motivated by the algorithmic Graph
Isomorphism problem. In the intervening years, CCs, and association schemes in
particular, have become basic objects of study in algebraic
combinatorics~\cite{bi-ac,bcn-drg,bailey-as,zieschang-as}.  
CCs also continue to play a role in the study of 
permutation groups~\cite{ep-transitive,mp-schur2group}.
Recent algorithmic applications of the CC concept 
include the Graph Isomorphism problem~\cite{babai-quasipolynomial}
and the complexity of matrix multiplication~\cite{cu-matrix}.

PCCs are in a sense the ``indivisible objects'' among CCs and are therefore of
particular interest.

In this paper we classify the PCCs with the largest automorphism groups, up to
the threshold stated in the following theorem.  (See
Defintion~\ref{def:exception} and Theorem~\ref{thm:main-pcc} for a more
detailed statement, and see Section~\ref{sec:asymptotic} for an explanation of
the asymptotic notation used throughout, including $O$, $\widetilde{O}$,
$\Theta$, $\Omega$, $\sim$, and $o$.)

\begin{theorem}\label{thm:main}
If $\mfX$ is a PCC not belonging to any of three exceptional families, then 
$|\Aut(\mfX)| \le
\exp(O(n^{1/3} \log^{7/3} n))$.
\end{theorem}

Primitive permutation groups of large order were classified
by Cameron~\cite{cameron-permutation}.  We refer to the 
orbital, or Schurian configurations of these groups as ``Cameron schemes''
(see Sections~\ref{sec:cc} and~\ref{sec:groups}).  For every $\epsilon>0$, for
every $n$ there is only a bounded number of primitive groups of order greater
than $\exp(n^{\epsilon})$ (the bound depends on $\epsilon$ but not on $n$);
we refer to this stratification as the ``Cameron hierarchy.''

Theorem~\ref{thm:main} represents progress on the following conjectured
classification of PCCs with large automorphism groups.

\begin{conjecture}[Babai]\label{conj:babai}
For every $\eps > 0$, there is some $N_\eps$ such that if $\mfX$ is a PCC on
$n \ge N_\eps$ vertices and $|\Aut(\mfX)| \ge \exp(n^\eps)$, then 
$\mfX$ is a Cameron scheme.  In particular, $\Aut(\mfX)$
is a known primitive group.
\end{conjecture}

This conjecture would be a far-reaching combinatorial generalization of
Cameron's classification of large primitive permutation groups.  In particular,
while Cameron's result is only known through the Classification of Finite
Simple Groups (CFSG), Conjecture~\ref{conj:babai} would imply (at least for orders
greater than $\exp(n^{\epsilon})$) a CFSG-free proof of Cameron's result,
giving a different kind of insight into the structure of large primitive
permutation groups.

Babai~\cite{babai-cc} established Conjecture~\ref{conj:babai}
for all $\eps > 1/2$ (the ``first level of the Cameron hierarchy'').  
As a corollary, he solved a then 100-year-old 
problem on primitive but not doubly transitive permutation groups, giving an 
nearly tight, $\exp(\wtO(n^{1/2}))$ bound on their order.  The tight bound
was subsequently found by Cameron, using CFSG; our result implies a
CFSG-free proof of the same tight bound.  Moreover,
our Theorem~\ref{thm:main} confirms the conjecture to all $\eps > 1/3$, the
first improvement since Babai's paper.   An elementary proof of 
an $\exp(\wtO(n^{1/3}))$ upper bound on the order of primitive 
permutation groups, with known exceptions (the ``second level
of the Cameron hierarchy'') follows.   

For the proof of Theorem~\ref{thm:main}, we find new combinatorial structure in
PCCs, including ``clique geometries'' in certain parameter ranges
(Theorem~\ref{thm:clique-geometry}).  An overview of our structural results for
PCCs is given in Section~\ref{sec:structure}.

Our motivation is thus twofold.   First, we develop a structure theory for
PCCs, the most general objects in a hierarchy of much-studied highly regular 
combinatorial structures.  Second, as a corollary to our main result, we
obtain a CFSG-free proof for the second level of the Cameron hierarchy
of large primitive permutation groups.

Additional motivation for our work comes from the algorithmic Graph Isomorphism problem. We explain this connection in Section~\ref{sec:isomorphism}.

\subsection{Exceptional coherent configurations}\label{sec:cc}

We now give a precise statement of our main combinatorial results.

Given a graph $X=(V,E)$, we associate with $X$ the configuration $\mfX(X) =
(V;\Delta,E,{\overline E})$ where ${\overline E}$ denotes the set of edges of
the complement of $X$.  (We omit $E$ if $E=\emptyset$ and omit ${\overline E}$
if ${\overline E}=\emptyset$.)  So graphs can be viewed as configurations of
rank $\le 3$.


Given an (undirected) graph $H$, the \emph{line-graph} $L(H)$ has as vertices the edges of
$H$, with two vertices adjacent in $L(H)$ if the corresponding edges are
incident in $H$.  The \emph{triangular graph} $T(m)$ is the line-graph of the
complete graph $K_m$ (so $n = \binom{m}{2}$). The \emph{lattice graph} $L_2(m)$
is the line-graph of the complete bipartite graph $K_{m,m}$ (on equal parts)
(so $n = m^2$). The configurations $\mfX(T(m))$ and $\mfX(L_2(m))$ are
coherent, and in fact primitive for $m > 2$.

\begin{definition}\label{def:exception}A PCC is \emph{exceptional} if it is of the form $\mfX(X)$,
    where $X$ is isomorphic to the complete graph $K_n$, the triangular graph
    $T(m)$, or the lattice graph $L_2(m)$, or the complement of such a graph.
\end{definition}

We note that the exceptional PCCs have $n^{\Omega(\sqrt{n})}$ automorphisms.
Indeed, the exceptional PCCs are exactly the ``orbital configurations'' of
large primitive permutation groups, as explained below.

Our main result is that all the non-exceptional PCCs have far fewer
automorphisms.

\begin{theorem}\label{thm:main-pcc}
If $\mfX$ is a non-exceptional PCC, then $|\Aut(\mfX)| \le
\exp(O(n^{1/3} \log^{7/3} n))$.
\end{theorem}

We remark that the bound of Theorem~\ref{thm:main-pcc} is tight, up to
polylogarithmic factors in the exponent. Indeed, the Johnson scheme
$J(m,3)$ and the Hamming scheme $H(3,m)$ both have $\exp(\Theta(n^{1/3}\log
n))$ automorphisms. (The Johnson scheme $J(m,3)$ has vertices the $3$-subsets
of a domain of size $m$ and $c(A,B) = |A\setminus B|$, and the Hamming scheme
$H(3,m)$ has vertices the words of length $3$ from an alphabet of size $m$ with
color $c$ given by the Hamming distance.)

The exceptional PCCs correspond naturally to the largest primitive permutation
groups. Given a permutation group $G \le \Sym(V)$, we define the \emph{orbital
configuration} $\mfX(G)$ on vertex set $V$ with the $R_i$ given by the orbitals
of $G$, i.e., the orbits of the induced action on $V\times V$.  CCs of this
form were first considered by Schur~\cite{schur-permutation}, and are commonly
called \emph{Schurian}. Note that $G \le \Aut(\mfX(G))$.

The Schurian CC $\mfX(G)$ is homogeneous if and only
if $G$ is transitive, and primitive if and only if $G$ is a primitive
permutation group.  If $G$ is doubly transitive, then $\mfX(G) =
\mfX(K_n)$. We also have $\mfX(S_m^{(2)}) = \mfX(T(m))$ and $\mfX(S_m\wr
S_2) = \mfX(L_2(m))$. 

\subsection{Primitive permutation groups}\label{sec:groups}
Following the completion of the Classification of Finite Simple Groups (CFSG),
one of the tasks has been to obtain elementary proofs of results currently
known only through CFSG. One such result is Cameron's classification of all
primitive permutation groups of large order, obtained by combining CFSG with the
O'Nan--Scott theorem~\cite{cameron-permutation}. Cameron's threshold for the
order is $n^{O(\log\log n)}$, but we state Mar\'oti's refinement of his
classification of permutation groups of order greater than $n^{c \log
n}$~\cite{maroti-primitive}.

\begin{theorem}[Cameron, Mar\'oti]\label{thm:cameron}
If $G$ is a primitive permutation group of degree $n > 24$, then one of the
following holds:
\begin{enumerate}
    \item[(a)] there are positive integers $d$, $k$, and $m$ such that
    $(A_m^{(k)})^d \le G \le S_m^{(k)} \wr S_d$;
    \item[(b)] $|G| \le n^{1 + \log_2 n}$.
\end{enumerate}
\end{theorem}

We call the primitive groups $G$ of Theorem~\ref{thm:cameron}~(a) \emph{Cameron
groups}.  Given a Cameron group $G$ with parameters $d$ and $k$ bounded, we
obtain a PCC $\mfX(G)$ with exponentially large automorphism group $H \ge G$,
in particular, of order $|H| \ge \exp(n^{1/(kd)})$. We call the PCCs $\mfX(G)$
\emph{Cameron schemes} when $G$ is a Cameron group. 

Hence, Conjecture~\ref{conj:babai} states that Cameron's classification of
primitive permutation groups transfers to the combinatorial setting of PCCs.
Furthermore, the conjecture entails Cameron's theorem, above the threshold $|G|
\ge \exp(n^{\eps})$ (see~\cite{bw-cameron}). Hence, confirmation of
Conjecture~\ref{conj:babai} would yield a CFSG-free proof of Cameron's
classification (above this threshold).

It seems unlikely that combinatorial methods will match Cameron's
$n^{O(\log\log n)}$ threshold for classification of primitive permutation
groups. An $n^{O(\log n)}$ threshold (as in stated Theorem~\ref{thm:cameron})
via elementary techniques might be possible, since above this threshold the
socle of a primitive permutation group is a direct product of alternating
groups, whereas below this threshold, simple groups of Lie type may appear in
the socle.

However, until the present paper, the only CFSG-free classification of the
large primitive permutation groups was given by Babai in a pair of papers in
1981 and 1982~\cite{babai-cc, babai-2transitive}.  Babai proved that $|G| \le
\exp(O(n^{1/2}\log^2 n))$ for primitive groups $G$ other than $A_n$ and
$S_n$~\cite{babai-cc}. A corollary of our work gives the first CFSG-free
improvement to Babai's bound, by proving that $|G| \le \exp(O(n^{1/3}\log^{7/3}
n))$ for primitive permutation groups $G$, other than groups belonging to three
exceptional families. 

In the following corollary to Theorem~\ref{thm:main}, $S_m^{(2)}$ and
$A_m^{(2)}$ denote the actions of $S_m$ and $A_m$, respectively, on the
$\binom{m}{2}$ pairs, and $G \wr H$ denotes the wreath product of the
permutation groups $G \le S_n$ by $H \le S_m$ in the product action on a domain
of size $n^m$.

\begin{corollary}\label{cor:main-group}
Let $\Gamma$ be a primitive permutation group of degree $n$.  Then either
$|\Gamma| \le \exp(O(n^{1/3}\log^{7/3} n))$, or $\Gamma$ is one of the following groups:
\begin{enumerate}
    \item[(a)] $S_n$ or $A_n$;
    \item[(b)] $S_m^{(2)}$ or $A_m^{(2)}$, where $n = \binom{m}{2}$;
    \item[(c)] a subgroup of $S_m \wr S_2$ containing $(A_m)^2$, where $n = m^2$.
\end{enumerate}
\end{corollary}

The slightly stronger bound $|\Gamma| \le \exp(O(n^{1/3}\log n))$ follows from
CFSG~\cite{cameron-permutation}. By contrast, our proof is elementary.

For given $n=m^2$, there are exactly three primitive groups in the third category
of Corollary~\ref{cor:main-group}.  We note that the groups of categories~1--3
of the corollary have order $\exp(\Omega(n^{1/2}\log n))$.

Corollary~\ref{cor:main-group} follows from Theorem~\ref{thm:main-pcc} by
classifying the large primitive groups $G$ for which $\mfX(G)$ is an
exceptional PCC, as in the following proposition. \qed

\begin{proposition}\label{prop:group-classify} There is a constant $c$ such
that the following holds. Let $G \le S_n$ be primitive, and suppose $|G| \ge
n^{c\log n}$.
\begin{enumerate}
    \item If $\mfX(G) = \mfX(K_n)$, then $G$ belongs to category~(a) of
    Corollary~\ref{cor:main-group}.
    \item If $\mfX(G) = \mfX(T(m))$, then $G$ belongs to category~(b) of
    Corollary~\ref{cor:main-group}.
    \item If $\mfX(G) = \mfX(L_2(m))$, then $G$ belongs to category~(c) of
    Corollary~\ref{cor:main-group}.
\end{enumerate}
\end{proposition}

Proposition~\ref{prop:group-classify} as stated requires CFSG, but an
elementary proof is available under the weaker bound of $|G| \ge \exp(c\log^3
n)$ using~\cite{pyber-2transitive}. For the proof and a more general classification,
we refer the reader to~\cite{bw-cameron}.

\subsection{Individualization and refinement}\label{sec:intro-refinement}

We now introduce the individualization/refinement heuristic.
We shall use individualization/refinement to find bases of automorphism groups
of configurations.

A \emph{base} for a group $G$ acting on a set $V$ is a subset $S \subseteq V$
such that the pointwise stablizer $G_{(S)}$ of $S$ in $G$ is trivial. If $S$ is
a base, then $|G| \le |V|^{|S|}$.

Let $\Iso(\mfX,\mfY)$ denote the set of isomorphisms from $\mfX$ to $\mfY$, and
$\Aut(\mfX) = \Iso(\mfX,\mfX)$.

\emph{Individualization} means the assignment of individual colors to some
vertices; then the irregularity so created propagates via some canonical color
refinement process.  For a class $\mcC$ of configurations (not necessarily
coherent), an assignment $\mfX \mapsto \mfX'$ is a \emph{color refinement} if
$\mfX, \mfX' \in \mcC$ have the same set of vertices and the coloring of
$\mfX'$ is a refinement of the coloring of $\mfX$.  Such an assignment is
\emph{canonical} if for all $\mfX,\mfY \in \mcC$, we have $\Iso(\mfX,\mfY) =
\Iso(\mfX',\mfY')$. In particular, $\Aut(\mfX) = \Aut(\mfX')$.

Repeated application of the refinement process leads to the \emph{stable
refinement} after at most $n-1$ rounds.

If after individualizing the elements of a set $S\subseteq V$, all vertices get
different colors in the resulting stable refinement, we say that $S$
\emph{completely splits} $\mfX$ (with respect to the given canonical refinement
process).  If $S$ completely splits $\mfX$, then $S$ is a base for
$\Aut(\mfX)$. Hence, to prove Theorem~\ref{thm:main-pcc}, it suffices to show
that some set of $O(n^{1/3}\log^{4/3} n)$ vertices completely splits $\mfX$
after canonical color refinement.

For our purposes, the simple ``naive vertex refinement'' will suffice as our
color refinement procedure.  Under \emph{naive vertex refinement}, the
edge-colors do not change, only the vertex-colors are refined. The refined
color of vertex $u$ of the configuration $\mfX$ encodes the following
information: the current color of $u$ and the number of vertices $v$ of color
$i$ such that $c(u,v)=j$, for every pair $(i,j)$, where $i$ is a vertex-color
and $j$ is an edge-color.

We now state our main technical result, from which Theorem~\ref{thm:main-pcc}
immediately follows.

\begin{theorem}[Main]\label{thm:main-ir}
Let $\mfX$ be a non-exceptional PCC. Then there exists a set of
$O(n^{1/3}\log^{4/3} n)$ vertices that completely splits $\mfX$ under naive
refinement.
\end{theorem}

This improves the main result of~\cite{babai-cc}, which stated that if $\mfX$
is a PCC other than $\mfX(K_n)$, then there is a set of $O(n^{1/2}\log n)$
vertices which completely splits $\mfX$ under naive refinement.

Naive vertex refinement is the only color refinement used in the present paper.
However, we remark that coherent configurations were first studied by
Weisfeiler and Leman in the context of their stronger canonical color
refinement~\cite{wl-wl,weisfeiler-graphs}.

Given a configuration $\mfX$, the \emph{Weisfeiler--Leman} (WL) canonical refinement
process~\cite{wl-wl, weisfeiler-graphs} produces a CC $\mfX'$ on the same
vertex set with $\Aut(\mfX) = \Aut(\mfX')$, by refining the coloring until it
is coherent.  More precisely, in every round of the refinement process, the
color $c(u,v)$ of the pair $u,v \in V$ is replaced with a color $c'(u,v)$ which
encodes $c(u,v)$ along with, for every pair $j,k$ of original colors, the
number of vertices $w$ such that $c(u,w) = i$ and $c(w,v) = k$. This refinement
is iterated until the coloring stabilizes, i.e., the rank no longer increases
in subsequent rounds of refinement.  The stable configurations under WL
refinement are exactly the coherent configurations.

\subsection{Relation to strongly regular graphs}

An undirected graph $X =(V,E)$ is called \emph{strongly regular} (SRG) with
parameters $(n,k,\lambda,\mu)$ if $X$ has $n$ vertices, every
vertex has degree $k$, each pair of adjacent vertices has $\lambda$
common neighbors, and each pair of non-adjacent vertices has $\mu$
common neighbors.

We note that a graph $X$ is a SRG if and only if the configuration $\mfX(X)$ is
coherent.  If a SRG $X$ is nontrivial, i.e., it is connected and coconnected,
then $\mfX(X)$ is a PCC.


All of our exceptional PCCs are in fact SRGs. Our classification of PCCs,
Theorem~\ref{thm:main-pcc}, was established in the special case of SRGs by
Spielman in 1996~\cite{spielman-srg}, on whose results we build. In fact, Chen,
Sun, and Teng have now established a stronger bound for SRGs: a non-exceptional
SRG has at most $\exp(\wtO(n^{9/37}))$ automorphisms~\cite{cst-srg}. 

The results of Spielman and Chen, Sun, and Teng both rely on Neumaier's
structure theory~\cite{neumaier-srg} of SRGs to separate the exceptional SRGs with many
automorphisms from those to which I/R can be effectively applied.  However, no
generalization of Neumaier's results to PCCs has been known. We provide a weak
generalization, sufficient for our purposes, in Section~\ref{sec:structure}.

\subsection{Graph Isomorphism}\label{sec:isomorphism}
The ``Graph Isomorphism (GI) problem'' is the computational problem 
to decide whether or not a pair of given graphs are isomorphic.
This problem is of great interest to complexity theory since it is
one of a very small number of natural problems in NP of intermediate
complexity status (unlikely to be NP-complete but not known to be
solvable in polynomial time).  

In recent major development, Babai~\cite{babai-quasipolynomial}
announced a quasipolynomial-time ($\exp(O((\log n)^c))$)
algorithm.

Babai's algorithm reduces the problem to the isomorphism problem 
of PCC's and then uses his (rather involved) ``split-or-Johnson''
procedure for further reduction.

Babai conjectures that a considerably simpler algorithm might succeed;
unless the PCC is a Cameron scheme,
individualization of a small number of vertices may
completely split the vertex set.  This is a more explicit
version of Conjecture~\ref{conj:babai}.

Our result proves that this is indeed the case if 
``small number'' means $\wtO(n^{1/3})$, improving
Babai's $\wtO(n^{1/2})$.  We hope that further refinement
of our structure theory will yield further progress
in this direction.

\subsection{Asymptotic notation}\label{sec:asymptotic}

To interpret asymptotic inequalities involving the parameters of a PCC, we
think of the PCC as belonging to an infinite family in which the asymptotic
inequalities hold. 

For functions $f,g : \bbN \to \bbR_{> 0}$, we write $f(n) = O(g(n))$ if there
is some constant $C$ such that $f(n) \le C g(n)$, and we write $f(n) =
\Omega(g(n))$ if $g(n) = O(f(n))$. We write $f(n) = \Theta(g(n))$ if $f(n) =
O(g(n))$ and $f(n) = \Omega(g(n))$. We use the notation $f(n) = \wtO(g(n))$
when there is some constant $c$ such that $f(n) = O(g(n)(\log n)^c)$. We write
$f(n) = o(g(n))$ if for every $\eps > 0$, there is some $N_{\eps}$ such that
for $n \ge N_{\eps}$, we have $f(n) < \eps g(n)$.  We write $f(n) =
\omega(g(n))$ if $g(n) = o(f(n))$. We use the notation $f(n) \sim g(n)$ for
asymptotic equality, i.e.,  $\lim_{n \to \infty} f(n)/g(n) =1$.  The asymptotic
inequality $f(n) \lesssim g(n)$ means $g(n) \sim \max\{f(n),g(n)\}$.


\section*{Acknowledgements}

The authors are grateful to L\'aszl\'o Babai for sparking our interest
in the problem addressed in this paper, providing insight into primitive
coherent configurations and primitive groups, uncovering a faulty application
of previous results in an early version of the paper, and giving invaluable
assistance in framing the results.


\section{Structure theory of primitive coherent configurations}\label{sec:structure}

To prove Theorem~\ref{thm:main-ir}, we need to develop a structure theory of
PCCs. The overview in this section highlights the main components of that
theory.

Throughout the paper, $\mfX$ will denote a PCC of rank $r$ on vertex set
$V$ with structure constants $p_{jk}^i$ for $0 \le i,j,k \le r-1$. \textbf{We
assume throughout that $r > 2$}, since the case $r=2$ is the trivial case of
$\mfX(K_n)$, listed as one of our exceptional PCCs. We also \textbf{assume}
without loss of generality that color $0$ corresponds to the diagonal, i.e.,
$R_0 = \{(u,u) : u \in V\}$.

For any color $i$ in a PCC, we write $n_i = n_{i^*} = p_{ii^*}^0 = p_{i^*i}^0$,
the out-degree of each vertex in $\mfX_i$.

We say that color $i$ is \emph{dominant} if $n_i \ge n/2$. Colors $i$ with $n_i
< n/2$ are \emph{nondominant}.  We call a pair of distinct vertices dominant
(nondominant) when its color is dominant (nondominant, resp.). We say color $i$
is \emph{symmetric} if $i^* = i$. Note that when color $i$ is dominant, it is
symmetric, since $n_{i^*} = n_i \ge n/2$.

Our analysis will divide into two cases, depending on whether or not there is a
dominant color. In fact, many of the results of this section will assume that
there is an overwhelmingly dominant color $i$ satisfying $n_i \ge n -
O(n^{2/3})$. The reduction to this case is accomplished via
Lemma~\ref{lem:dominant} of the next section. The main structural result used
in its proof is Lemma~\ref{lem:growth-spheres} below, which gives a lower bound
on the growth of ``spheres'' in a PCC.

For a color $i$ and vertex $u$, we denote by $\mfX_i(u)$ the
set of vertices $v$ such that $c(u,v) = i$.  We denote by $\dist_i(u, v)$ the
directed distance from $u$ to $v$ in the color-$i$ constituent digraph
$\mfX_i$, and we write $\dist_i(j) = \dist_i(u,v)$ for any vertices $u,v$ with
$c(u,v) = j$. (This latter quantity is well-defined by the coherence of
$\mfX$.) The \emph{$\delta$-sphere} $\mfX_i^{(\delta)}(u)$ in $\mfX_i$ centered
at $u$ is the set of vertices $v$ with $\dist_i(u,v) = \delta$. 

\begin{lemma}[Growth of spheres]\label{lem:growth-spheres}
    Let $\mfX$ be a PCC, let $i,j \ge 1$ be nondiagonal colors, let $\delta =
    \dist_i(j)$, and $u \in V$. Then for any integer $1 \le \alpha \le
    \delta-2$, we have 
    \[
    |\mfX_i^{(\alpha+1)}(u)||\mfX_i^{(\delta-\alpha)}(u)| \ge
    n_in_j.
    \]
\end{lemma}

We note that Lemma~\ref{lem:growth-spheres} is straightforward when $\mfX_i$
is distance-regular. Indeed, a significant portion of the difficulty of the
lemma was in finding the correct generalization. 

\begin{proof}[Overview of proof of Lemma~\ref{lem:growth-spheres}]
The bipartite subgraphs of $\mfX_i$ induced on pairs of the form
$(\mfX_j(u),\mfX_k(u))$, where $j,k$ are colors and $u$ is a vertex, are
biregular by the coherence of $\mfX_i$. We exploit this biregularity to count
shortest paths in $\mfX_i$ between a carefully chosen subset of
$\mfX_i^{(\delta-\alpha)}(u)$ and $\mfX_j(u)$, for an arbitrary vertex $u$.

The details of the proof are given in Section~\ref{sec:spheres}.
\end{proof}

In the rest of the paper, we \textbf{assume} without loss of generality that
$n_1 = \max_i n_i$. We write $\rho = \sum_{i \ge 2} n_i = n -n_1 -1$. For the
rest of the section, color $1$ will in fact be dominant. In fact, every theorem
in the rest of this section will state the assumption that $\rho = o(n^{2/3})$.
Lemma~\ref{lem:diam2} below demonstrates some of the power of this supposition.

\begin{lemma}\label{lem:diam2}
    Let $\eps > 0$ and let $\mfX$ be a PCC with $\rho <  (1-\eps)n^{2/3})$.
    Then, for $n$ sufficiently large, $\dist_i(1) = 2$ for every nondominant
    color $i$. Consequently, $n_i \ge \sqrt{n-1}$ for $i \ne 0$.
\end{lemma}
\begin{proof}[Overview of proof of Lemma~\ref{lem:diam2}]
    We will prove that if $\dist_i(1) \ge 3$ for some color $i$, then
    $\rho \gtrsim n^{2/3}$. Without loss of generality, we assume $n_1 \sim n$,
    since otherwise we are already done.

    Fix an arbitrary vertex $u$ and consider the bipartite graph $B$ between
    $\mfX_i^{(\delta - 1)}(u)$ and $\mfX_1(u)$, with an edge from $x \in
    \mfX_i^{(\delta - 1)}(u)$ to $y \in \mfX_1(u)$ when $c(x, y) = i$.  By the
    coherence of $\mfX$, the bipartite graph is regular on $\mfX_1(u)$; call
    its degree $\gamma$. An obstacle to our analysis is that the graph need not
    be biregular.  Nevertheless, we estimate the maximum degree $\beta$ of a
    vertex in $\mfX_i^{(\delta-1)}(u)$ in $B$. We first note that $n_1 \gamma
    \leq \beta \rho$.

    Let $w$ be a vertex satisfying $c(u, w) = i$. We pass to the subgraph $B'$
    induced on $(\mfX_i^{(\delta-2)}(w),\mfX_i^{(\delta-1)}(w))$, and observe
    that the degree of vertices in $\mfX_i^{(\delta-1)}(u)\cap
    \mfX_i^{(\delta-2)}(w)$ is preserved, while the degree of vertices in
    $\mfX_1(u)\cap \mfX_i^{(\delta-1)}(w)$ does not increase.  Let
    $v$ be a vertex of degree $\beta$ in $B'$, and let $j = c(w,v)$. We finally
    consider the bipartite graph $B''$ on $(\mfX_j(w),X_w)$, where $X_w$ is the
    set of vertices $x \in \mfX_i^{(\delta-1)}(w)$ with at most $\gamma$
    in-neighbors in $\mfX_i$ lying in the set $\mfX_j(w)$. In particular,
    $\mfX_1(u)\cap \mfX_i^{(\delta-1)}(w) \subseteq X_w$. This graph $B''$ is
    now regular (of degree $\ge \beta$) on $\mfX_j(w)$. Since $X_w \subseteq
    \mfX_i^{(\delta-1)}(w)$, we have $|X_w| \le\rho$, which eventually
    gives the bound $\beta \le \gamma\rho^2/n_1$. Combining this with our
    earlier estimate $\beta \rho \ge n_1 \gamma$ proves the lemma.

    The details of the proof are given in Section~\ref{sec:diam2}.
\end{proof}

\vspace{10pt}
\emph{Notation.} Let $G(\mfX)$ be the graph on $V$ formed by the nondominant pairs. So $G(\mfX)$
is regular of valency $\rho$, and every pair of distinct nonadjacent vertices
in $G(\mfX)$ has exactly $\mu$ common neighbors, where $\mu = \sum_{i,j \ge
2}p^1_{ij}$.  The graph $G(\mfX)$ is not generally SR, since pairs of adjacent
vertices in $G(\mfX)$ of different colors in $\mfX$ will in general have
different numbers of common neighbors. However, intuition from SRGs will
prove valuable in understanding $G(\mfX)$.

We write $N(u)$ for the set of neighbors of $u$ in the graph $G(\mfX)$.  For
$i$ nondominant, we define $\lambda_i = |\mfX_i(u) \cap N(v)|$, where $c(u,v) =
i$. So, the parameters $\lambda_i$ are loosely analogous to the parameter
$\lambda$ of a SRG.

A \emph{clique} $C$ in an undirected graph $G$ is a set of pairwise adjacent
vertices; its \emph{order} $|C|$ is the number of vertices in the set.

\begin{definition}\label{def:clique-geom}
    A \emph{clique geometry} on a graph $G$ is a collection $\mcG$ of maximal
    cliques such that every pair of adjacent vertices in $G$ belongs to a
    unique clique in $\mcG$. A clique geometry of a PCC $\mfX$ is a clique
    geometry on $G(\mfX)$.  The clique geometry $\mcG$ is
    \emph{asymptotically uniform} (for an infinite family of PCCs) if for
    every $C \in \mcG$, $u \in C$, and nondominant color $i$, we have either
    $|C \cap \mfX_i(u)| \sim \lambda_i$ or $|C \cap \mfX_i(u)| = 0$ (as $n \to
    \infty$).
\end{definition}

We have the following sufficient condition for the existence of clique
geometries in PCCs.

\begin{theorem}\label{thm:clique-geometry}
    Let $\mfX$ be a PCC satisfying $\rho = o(n^{2/3})$, and fix a constant
    $\eps > 0$. If $\lambda_i \ge \eps n^{1/2}$ for every nondominant color
    $i$, then for $n$ sufficiently large, there is a clique geometry $\mcG$ on $\mfX$.
    Moreover, $\mcG$ is asymptotically uniform.
\end{theorem}

Theorem~\ref{thm:clique-geometry} provides a powerful dichotomy for PCCs:
either there is an upper bound on some parameter $\lambda_i$, or there is a
clique geometry.  Adapting a philosophy expressed in~\cite{bcstw-srg}, we note
that bounds on $\lambda_i$ are useful because they limit the correlation
between the $i$-neighborhoods of two random vertices. Similar bounds on the
parameter $\lambda$ of a SRG were used in~\cite{bcstw-srg}.

On the other hand, Theorem~\ref{thm:clique-geometry} guarantees that if all
parameters $\lambda_i$ are sufficiently large, the PCC has an asymptotically
uniform clique geometry. This is our weak analogue of Neumaier's geometric
structure. Clique geometries offer their own dichotomy. Geometries with at most
two cliques at a vertex can classified; this includes the exceptional PCCs
(Theorem~\ref{thm:2clique-characterization} below). A far more rigid structure
emerges when there are at least three cliques at every vertex. In this case, we
exploit the ubiquitous 3-claws (induced $K_{1,3}$ subgraphs) in $G(\mfX)$ in
order to construct a set which completely splits $\mfX$
(Lemma~\ref{lem:main-spielman-combined}~(b)).

\begin{proof}[Overview of proof of Theorem~\ref{thm:clique-geometry}]
    The existence of a weaker clique structure follows from a result of
    Metsch~\cite{metsch-clique}.
    (See Lemma~\ref{lem:metsch-local} below and the
    comments in the paragraph preceding it.) Specifically, under the hypotheses
    of Theorem~\ref{thm:clique-geometry}, for every nondominant color $i$ and
    vertex $u$, there is a partition of $\mfX_i(u)$ into cliques of order $\sim
    \lambda_i$ in $G(\mfX)$.  We call such a collection of cliques a
    \emph{local clique partition} (referring to the color-$i$ neighborhood of
    any fixed vertex).
    
    The challenge is to piece together these local clique partitions into a
    clique geometry. An obstacle is that Metsch's cliques are cliques of
    $G(\mfX)$, not $\mfX_i$; that is, the edges of the cliques partitioning
    $\mfX_i(u)$ have nondominant colors but not in general color $i$.
    In particular, for two vertices $u,v \in V$ with $c(u,v) = i$, the clique
    containing $v$ in the partition of $\mfX_i(u)$ may not correspond to any of
    the cliques in the partition of $\mfX_i(v)$. 

    We first generalize these local structures.  An \emph{$I$-local clique
    partition} is a partition of $\bigcup_{i \in I}\mfX_i(u)$ into cliques of
    order $\sim \sum_{i \in I}\lambda_i$. We study the maximal sets $I$ for
    which such $I$-local clique partitions exist, and eventually prove that
    these maximal sets $I$ partition the set of nondominant colors, and the
    corresponding cliques are maximal in $G(\mfX)$.

    Finally, we prove a symmetry condition: given a nondominant pair of
    vertices $u,v \in V$, the maximal local clique at $u$ containing
    $v$ is equal to the maximal local clique at $v$ containing $u$.  This
    symmetry ensures the cliques form a clique geometry, and this clique
    geometry is asymptotically uniform by construction.

    The details of the proof are given in Section~\ref{sec:clique}.
\end{proof}

The case that $\mfX$ has a clique geometry with some vertex belonging to at
most two cliques includes the exceptional CCs corresponding to $T(m)$ and
$L_2(m)$. We give the following classification.

\begin{theorem}\label{thm:2clique-characterization}
    Let $\mfX$ be a PCC such that $\rho = o(n^{2/3})$. Suppose that
    $\mfX$ has an asymptotically uniform clique geometry $\mcG$ and a vertex $u \in V$ belonging to at
    most two cliques of $\mcG$. Then for $n$ sufficiently large, one of the
    following is true:
\begin{enumerate}
    \item[(a)] $\mfX$ has rank three and is isomorphic to $\mfX(T(m))$ or
    $\mfX(L_2(m))$;
    \item[(b)] $\mfX$ has rank four, $\mfX$ has a non-symmetric non-dominant
        color $i$, and $G(\mfX)$ is isomorphic to $T(m)$ for $m = n_i+2$.
\end{enumerate}
\end{theorem}
\begin{proof}[Overview of proof of Theorem~\ref{thm:2clique-characterization}]
    We first use the coherence of $\mfX$ to show that every vertex $u \in V$
    belongs to exactly two cliques of $\mcG$, and these cliques have
    order $\sim \rho/2$. By counting vertex-clique incidences, we then obtain the
    estimate $\rho \lesssim 2\sqrt{2n}$.  On the other hand, by
    Lemma~\ref{lem:diam2}, every nondominant color $i$ satisfies $n_i \gtrsim
    \sqrt{n}$. Hence, there are at most $2$ nondominant colors.

    Since every vertex belongs to exactly two cliques, the graph $G(\mfX)$ is
    the line-graph of a graph. If there is only one nondominant color, then
    $G(\mfX)$ is strongly regular, and therefore, for $n$ sufficiently large,
    $G(\mfX)$ is isomorphic to $T(m)$ or $L_2(m)$. On the other hand, if
    there are two nondominant colors, by counting paths of length $2$ we show
    that $G(\mfX)$ must again be isomorphic to $T(m)$. By studying the
    edge-colors at the intersection of the cliques containing two distinct
    vertices and exploiting the coherence of $\mfX$, we finally eliminate the
    case that the two nondominant colors are symmetric.
    
    The details of the proof are given in Section~\ref{sec:2clique}.
\end{proof}


\section{Overview of analysis of I/R}\label{sec:overview}

We now give a high-level overview of how we apply our structure theory of PCCs
to prove Theorem~\ref{thm:main-ir}.

Most of the results highlighted in Section~\ref{sec:structure} assumed
that $\rho = o(n^{2/3})$. Hence, the first step is to reduce to this case,
which we accomplish via the following lemma.

\begin{lemma}\label{lem:dominant}
    Let $\mfX$ be a PCC. If $\rho \ge  n^{2/3} (\log n)^{-1/3}$, then there is
    a  set of size $O(n^{1/3}(\log n)^{4/ 3})$ which completely splits $\mfX$.
\end{lemma}

We remark that in the case that the rank $r$ of $\mfX$ is bounded, our
Lemma~\ref{lem:dominant} follows from a theorem of Babai~\cite[Theorem 2.4]{babai-cc}.
Following Babai~\cite{babai-cc}, we analyze the distinguishing number.
\begin{definition} Let $u,v \in V$. We say $w \in V$ \emph{distinguishes} $u$
and $v$ if $c(w,u) \ne c(w,v)$. We write $D(u,v)$ for the set of vertices $w$
distinguishing $u$ and $v$, and $D(i) = |D(u,v)|$ where $c(u,v) = i$. We call
$D(i)$ the \emph{distinguishing number} of $i$.
\end{definition}
Hence, $D(i) = \sum_{j \ne k} p^i_{jk^*}$. If $w \in D(u,v)$, then after
individualizing $w$ and refining, $u$ and $v$ get different colors.

Babai observed that in order to completely split a PCC $\mfX$, it suffices to
individualize some set of $O(n \log n/D_{\min})$ vertices, where $D_{\min} =
\min_{i\ne 0}\{D(i)\}$~\cite[Lemma 5.4]{babai-cc}.  Thus, to prove
Lemma~\ref{lem:dominant}, we show that if $\rho \ge n^{2/3}(\log n)^{-1/3}$
then for every color $i \ne 0$, we have $D(i) = \Omega(n^{2/3}  (\log
n)^{-1/3})$. 

The following bound on the number of large colors in a PCC becomes powerful
when $D(i)$ is small.

\begin{lemma}\label{lem:few-large-colors}
    Let $\mfX$ be a PCC. For any nondiagonal color $i$, the number of colors
    $j$ such that $n_j > n_i/2$ is at most $O((\log n +n/\rho)D(i)/n_i)$.
\end{lemma}
\begin{proof}[Overview of proof of Lemma~\ref{lem:few-large-colors}]
    Let $I_{\alpha}$ be the set of colors $i$ such that $D(i) \le \alpha$,
    and $J_{\beta}$ the set of colors $j$ such that $n_j \le \beta$.
    For a set $I$ of colors, let $n_I = \sum_{i \in I}n_i$ be the total degree
    of the colors in $I$.

    First, we prove that $\lfloor \alpha/(3D(i))\rfloor n_i \le
    n_{I_{\alpha}}$, a lower bound on the total degree of colors with distinguishing
    number $\le \alpha$. Next, we prove a lemma that allows us to transfer
    estimates for the total degree of colors with small distinguishing number
    into estimates for the total degree of colors with low degree.
    Specifically, we prove that $n_{J_{\beta}} \le 2\alpha$, where $\beta =
    n_{I_{\alpha}/2}$. Together, these two results allow us to transfer
    estimates on total degree between the sets $I_{\alpha}$ and $J_{\beta}$, as
    $\alpha$ and $\beta$ increase. 

    The details of the proof are given in Section~\ref{sec:dominant}.
\end{proof}

\begin{proof}[Overview of proof of Lemma~\ref{lem:dominant}]
    Fix a color $i \ge 1$. We wish to give a lower bound on $D(i)$. Babai
    observed that for any color $j \ge 1$, we have $D(i) \ge
    D(j)/\dist_i(j)$~\cite[Proposition 6.4 and Theorem 6.1]{babai-cc}). Hence,
    we wish to give an upper bound on $\dist_i(j)$ for some color $j$ with
    $D(j)$ large.

We analyze two cases: $n^{2/3}(\log n)^{-1/3} \le \rho < n/3$ and $\rho
\ge n/3$.

In the former case, when $n^{2/3}(\log n)^{-1/3} \le \rho < n/3$, we first
observe that $D(1) = \Omega(\rho)$.   Hence, the problem is reduced to bounding
the quantity $\dist_i(1)$ for every color $i$. Our bound in
Lemma~\ref{lem:growth-spheres} on the size of spheres suffices for this task
since $n_1$ is large.

In the case that $\rho \ge n/3$, Babai observed that the color $j$ maximizing
$D(j)$ satisfies $D(j) = \Omega(n)$.  We partition the colors of $\mfX$
according to their distinguishing number, by first partitioning the positive
integers less than $D(j)$ into cells of length $3D(i)$. (Specifically, we
partition the colors of $\mfX$ so that the $\alpha$-th cell contains the colors
$k$ satisfying $3D(i)\alpha \le D(k) < 3D(i)(\alpha+1)$, and there are $O(D(j)/D(i))$
cells.) Each cell of this partition $\mcP$ is nonempty. In fact, we show that
the sum of the degrees of the colors in each cell is at least $n_i$.

On the other hand, Lemma~\ref{lem:few-large-colors} says that there are few
colors $k$ satisfying $n_k > n_i/2$, and we show that the total degree of the colors
$k$ with $n_k \le n_i/2$ is also small.  Since each cell of the partition has
degrees summing to at least $n_i$, these together give an upper bound on the
number of cells, and hence a lower bound on $D(i)$.

The details of the proof are given in Section~\ref{sec:dominant}.
\end{proof}

We have now reduced to the case that $\rho = o(n^{2/3})$. Our analysis of this
case is inspired by Spielman's analysis of SRGs~\cite{spielman-srg}.

\begin{lemma}\label{lem:main-spielman-combined}
    There exists a constant $\eps > 0$ such that the following holds. Let
    $\mfX$ be a PCC with $\rho = o(n^{2/3})$. If $\mfX$ satisfies either of the following
    conditions, then there is a set of $O(n^{1/4}(\log n)^{1/2})$
    vertices which completely splits $\mfX$.
    \begin{enumerate}
        \item[(a)] There is a nondominant color $i$ such that $\lambda_i < \eps
            n^{1/2}$.
        \item[(b)] For every nondominant color $i$, we have $\lambda_i \ge \eps
            n^{1/2}$. Furthermore, $\mfX$ has an asymptotically uniform clique
            geometry $\mcC$ such that every vertex belongs to at least three
            cliques of $\mcC$.
    \end{enumerate}
\end{lemma}
\begin{proof}[Overview of proof of Lemma~\ref{lem:main-spielman-combined}]
    We will show that if we individualize a random set of
    $O(n^{1/4}(\log n)^{1/2})$ vertices, then with positive probability, every
    pair of distinct vertices gets different colors in the stable refinement.
    
    Let $u,v \in C$, and fix two colors $i$ and $j$. Generalizing a pattern
    studied by Spielman, we say a triple $(w,x,y)$ is good for $u$ and $v$ if
    $c(u, x) = c(u, y ) = c(x,y) =1$, $c(u,w) = i$, and $c(w,x) = c(w,y) = j$,
    but there exists no vertex $z$ such that $c(v,z) = i$ and $c(z,x) = c(z,y)
    = j$. (See Figure~\ref{fig:spielman}). To ensure that $u$
    and $v$ get different colors in the stable refinement, it suffices to
    individualize two vertices $x,y \in V$ for which there exists a vertex $w$
    such that $(w,x,y)$ is good for $u$ and $v$.  We show that if there are
    many good triples for $u$ and $v$, then individualizing a random set of
    $O(n^{1/4}(\log n)^{1/2})$ vertices is overwhelming likely to result in the
    individualization of such a pair $x,y \in V$.

    Condition~(a) of the lemma is analogous to the asymptotic consequences of
    Neumaier's claw bound used by Spielman~\cite{spielman-srg} (cf.~\cite[Section
    2.2]{bw-delsarte}), except that the bound on $\lambda_i$ does not imply a
    similar bound on $\lambda_{i^*}$. We show that a relatively weak bound on
    $\lambda_{i^*}$ already suffices for Spielman's argument to essentially go
    through. However, if even this weaker assumption fails, then we turn to our
    local clique structure for the analysis (as described in the overview
    of Theorem~\ref{thm:clique-geometry}).
    
    When condition~(b) holds, we cannot argue along Spielman's lines, and
    instead analyze the structural properties of our clique geometries to
    estimate the number of good triples.

    The details of the proof are given in Section~\ref{sec:spielman}.
\end{proof}

By Theorem~\ref{thm:clique-geometry}, either the hypotheses of of
Lemma~\ref{lem:main-spielman-combined} are satisfied, or $\mfX$ has an
asymptotically uniform clique geometry $\mcC$, and some vertex belongs to at
most two cliques of $\mcC$.  Theorem~\ref{thm:2clique-characterization} gives a
characterization PCCs $\mfX$ with the latter property: $\mfX$ is one of the
exceptional PCCs, or $\mfX$ has rank four with a non-symmetric non-dominant
color $i$ and $G(\mfX)$ is isomorphic to $T(m)$ for $m = n_i +2$. We handle
this final case via the following lemma, proved in
Section~\ref{sec:2clique}. 

\begin{lemma}\label{lem:2clique-split}
    Let $\mfX$ be a PCC satisfying
    Theorem~\ref{thm:2clique-characterization}~(b). Then some set of size
    $O(\log n)$ completely splits $\mfX$.
\end{lemma}

We conclude this overview by observing that Theorem~\ref{thm:main-ir} follows
from the above results.

\begin{proof}[Proof of Theorem~\ref{thm:main-ir}]
    Let $\mfX$ be a PCC. Suppose first that $\rho \geq n^{2/3} (\log
    n)^{-1/3}$. Then by Lemma~\ref{lem:dominant}, there is a set of size
    $O(n^{1/3}(\log n)^{4/ 3})$ which completely splits $\mfX$. 
    
    Otherwise, $\rho < n^{2/3} (\log n)^{-1/3} = o(n^{2/3})$. By
    Theorem~\ref{thm:clique-geometry}, either the hypotheses of
    Lemma~\ref{lem:main-spielman-combined} are satisfied, or the hypotheses of
    Theorem~\ref{thm:2clique-characterization} are satisfied. In the former
    case, some set of $O(n^{1/4}(\log n)^{1/2})$ vertices completely splits
    $\mfX$. In the latter case, either $\mfX$ is exceptional, or, by
    Lemma~\ref{lem:2clique-split}, some set of $O(\log n)$ vertices completely
    splits $\mfX$.
\end{proof}


\section{Growth of spheres}\label{sec:spheres}

In this section, we will prove Lemma~\ref{lem:growth-spheres}, our estimate of
the size of spheres in constituent digraphs.

We start from a few basic observations.

\begin{proposition}\label{prop:biregular}
    Let $G = (A,B,E)$ be a bipartite graph, and let $A_1\cup\cdots\cup A_m$ be
    a partition of $A$ such that the subgraph induced on $(A_i,B)$ is biregular
    of positive valency for each $1 \le i \le m$. Then for any $A' \subseteq
    A$, we have \[|N(A')|/|A'| \ge |B|/|A|\] where $N(A')$ is the set of
    neighbors of vertices in $A'$, i.e., $N(A') = \{y \in B : \exists
    x \in A', \{x,y\} \in E\}$.
\end{proposition}
\begin{proof}
    Let $A' \subseteq A$. By the pigeonhole principle, there is some $i$ such
    that $|A'\cap A_i|/|A_i| \ge |A'|/|A|$. Let $\alpha$ be the degree of a
    vertex in $A_i$ and let $\beta$ be the number of neighbors in $A_i$ of a
    vertex in $B$.  We have $\alpha|A_i| = \beta|B|$, and $\beta|N(A' \cap
    A_i)|\ge \alpha|A' \cap A_i|$. Hence,
    \[
        |N(A')| \ge |N(A'\cap A_i)| \ge \frac{|A' \cap A_i|\alpha}{\beta} = \frac{|A'\cap A_i| |B|}{|A_i|} \ge
        \frac{|B||A'|}{|A|}. \qedhere
    \]
\end{proof}

Suppose $A, B \subseteq V$ are disjoint set of vertices.  We denote by $(A, B,
i)$ the bipartite graph between $A$ and $B$ such that there is an edge from $x
\in A$ to $y \in B$ if $c(x, y) = i$.  For $I \subseteq [r - 1]$ a set of nondiagonal 
colors,  we denote by $(A, B, I)$ the bipartite graph between $A$ and $B$ such
that there is an edge from $x \in A$ to $y \in B$ if $c(x, y) \in I$.  

\begin{fact}\label{fact:bipartite_regular}
For any vertex $u$, colors $0 \leq j, k \leq r -  1$ with $j \neq k$, and set
$I \subseteq [r-1]$ of nondiagonal colors, the bipartite graph $(\mfX_j(u), \mfX_k(u),
I)$ is biregular.
\end{fact}
\begin{proof}
The degree of every vertex in $\mfX_j(u)$ is $\sum_{i \in I}p^j_{i k^*}$.
And the degree of every vertex in $\mfX_k(u)$ is $\sum_{i \in I}p^k_{j i}$.
\end{proof}

Recall our notation $\mfX_i^{(\delta)}(u)$ for the $\delta$-sphere centered at
$u$ in the color-$i$ constituent digraph, i.e., the set of vertices $v$ such
that $\dist_i(u,v)=\delta$. 

For the remainder of Section~\ref{sec:spheres}, we fix a PCC $\mfX$, a color $1
\le i \le r-1$, and a vertex $u$. For a color $1 \le j \le r-1$ and an integer
$1 \leq \alpha \leq \dist_i(j)$, we denote by $S_\alpha^{(j)}$ the set of
vertices $v \in \mfX_i^{(\alpha)}(u)$ such that there is a vertex $w \in
\mfX_j(u)$ and a shortest path in $\mfX_i$ from $u$ to $w$ passing through $v$,
i.e.,
\[
    S_\alpha^{(j)} = \{v \in \mfX_i^{(\alpha)}(u): \exists w \in \mfX_j(u) \text{ s.t. }
\dist_i(u, v) + \dist_i(v, w) = \dist_i(u, w)\}.
\]
Note that these sets $S_{\alpha}^{(j)}$ are nonempty by the primitivity of
$\mfX$, and in particular, if $\alpha = \dist_i(j)$, then $S_\alpha^{(j)} =
\mfX_j(u)$.  For $v \in V$ and an integer $\dist_i(u, v) < \alpha \leq
\dist_i(j)$, we denote by $S_\alpha^{(j)}(v) \subseteq S_\alpha^{(j)}$ the set
of vertices $x \in S_\alpha^{(j)}$ such that there is a shortest path in
$\mfX_i$ from $u$ to $x$ passing through $v$, i.e. 
\begin{align*}
S_\alpha^{(j)}(v) &= S_\alpha^{(j)}\cap \mfX_i^{(\alpha-\dist_i(u,v))}(v) \\
&= \{x \in S_\alpha^{(j)} : \dist_i(u, v) + \dist_i(v, x) = \dist_i(u, x)\}.
\end{align*}

See Figure 1 for a graphical explanation of the
notation. 

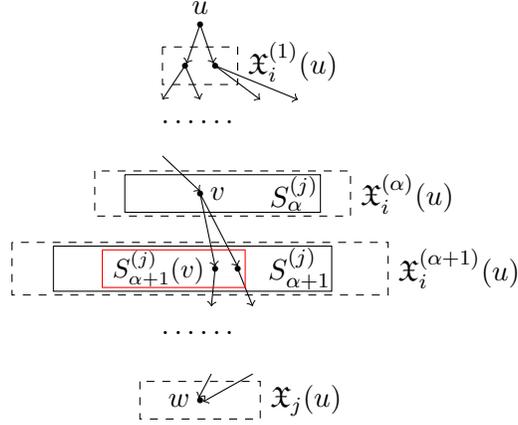
\begin{figure}
\begin{center}

\begin{tikzpicture}

\def\ptrad{0.04}

\definecolor{plum}{RGB}{142, 69, 33}

\coordinate (u) at (0,0);
\coordinate (w) at (0, -5);
\coordinate (v) at (0,-2.25);
\coordinate (x1) at (-0.2, -0.55);
\coordinate (x2) at (0.2, -0.55);
\coordinate (y1) at (0.2, -3.25);
\coordinate (y2) at (0.5, -3.25);
\coordinate (z1) at (0.15, -3.75);
\coordinate (z2) at (0.7, -3.75);
\coordinate (a1) at (0.15, -4.65);
\coordinate (a2) at (0.7, -4.65);

\fill[color = black] (u) circle[radius=\ptrad];
\fill[color = black] (v) circle[radius=\ptrad];
\fill[color = black] (w) circle[radius=\ptrad];
\fill[color = black] (x1) circle[radius=\ptrad];
\fill[color = black] (x2) circle[radius=\ptrad];
\fill[color = black] (y1) circle[radius=\ptrad];
\fill[color = black] (y2) circle[radius=\ptrad];

\node[above] at (u) {$u$};
\node[left] at (w) {$w$};
\node[right] at (v) {$v$};
\draw[->](u) --(-0.18, -0.52);
\draw[->](u) --(0.18, -0.52);
\draw[->](x1) -- (0, -1);
\draw[->](x1) -- (-0.5, -1);
\draw[->](x2) -- (0.8, -1);
\draw[->](x2) -- (1.3, -1);

\draw[->](-0.5, -1.75) -- (0, -2.22);
\draw[->](v) -- (0.2, -3.22);
\draw[->](v) -- (0.5, -3.22);

\draw[->](y1) -- (z1);
\draw[->](y2) -- (z2);

\node at (0, -4.1) {$\dots \dots$};

\draw[->](a1) -- (-0.02, -4.97);
\draw[->](a2) -- (0.04, -5.02);

\draw[dashed] (-0.5, -0.3) -- (-0.5, -0.8) -- (0.5, -0.8) -- (0.5, -0.3) -- (-0.5, -0.3);
\node[right] at (0.5, -0.5) {$\mfX_i^{(1)}(u)$};

\node at (0, -1.3) {$\dots \dots$};

\draw[dashed] (-1.4, -1.95) -- (-1.4, -2.55) -- (2, -2.55) -- (2, -1.95) -- (-1.4, -1.95);
\draw (-1, -2) -- (-1, -2.5) -- (1.6, -2.5) -- (1.6, -2) -- (-1, -2);
\node[right] at (2, -2.25) {$\mfX_i^{(\alpha)}(u)$};
\node at (1.25, -2.25) {$S^{(j)}_{\alpha}$};

\draw[dashed] (-2.5, -2.9) -- (-2.5, -3.6) -- (2.5, -3.6) -- (2.5, -2.9) -- (-2.5, -2.9);
\draw (-1.95, -2.95) -- (-1.95, -3.55) -- (1.75, -3.55) -- (1.75, -2.95) -- (-1.95, -2.95);
\draw[color = red](-1.3, -3) -- (-1.3, -3.5) -- (0.6, -3.5) -- (0.6, -3) -- (-1.3, -3);
\node[right] at (2.5, -3.25) {$\mfX_i^{(\alpha+1)}(u)$};

\node[right] at (-1.3, -3.25) {\small $S^{(j)}_{\alpha+1}(v)$};
\node[right] at (0.77, -3.25) {$S^{(j)}_{\alpha+1}$};

\draw[dashed] (-0.8, -4.75) -- (-0.8, -5.25) -- (0.8, -5.25) -- (0.8, -4.75) -- (-0.8, -4.75);
\node[right] at (0.8, -5) {$\mfX_j(u) $};





\end{tikzpicture}
\end{center}
\caption{$S_\alpha^{(j)}$ and $S_{\alpha+1}^{(j)}(v)$.}
\label{fig:distinguish-notation}
\end{figure}

\begin{corollary}\label{cor:biregular-applied}
    Let $1 \le j \le r-1$ be a color such that $\delta = \dist_i(j)
    \ge 3$. Let $1 \leq \alpha \le \delta-2$ be an integer, and let $v \in S_\alpha^{(j)}$. Then
    \[
        \frac{|S_\delta^{(j)}(v)|}{|S_{\alpha+1}^{(j)}(v)|} \ge
        \frac{n_j}{|S_{\alpha+1}^{(j)}|}.
    \]
\end{corollary}
\begin{proof}
    Consider the bipartite graph $(S_{\alpha+1}^{(j)},\mfX_j(u), I)$ with 
    \[I = \{k: 1 \leq k \leq r - 1 \text{ and } \dist_i(k) = \dist_i(j) - \alpha - 1\}.\]
    There is an edge from $x \in S_{\alpha + 1}^{(j)}$ to $y \in \mfX_j(u)$ if there
    is a shortest path from $u$ to $y$ passing through $x$.

    By the coherence of $\mfX$, if $\mfX_\ell(u)\cap S_{\alpha+1}^{(j)}$ is nonempty
    for some color $\ell$, then $\mfX_\ell(u)\subseteq S_{\alpha+1}^{(j)}$. Hence,
    $S_{\alpha+1}^{(j)}$ is partitioned into sets of the form $\mfX_\ell(u)$ with
    $\dist_i(\ell) = \alpha + 1$.  For such colors $\ell$, by Fact
    \ref{fact:bipartite_regular},  $(\mfX_\ell(u), \mfX_j(u), I)$ is biregular,
    and by the definition of $S^{(j)}_{\alpha+1}$, then $(\mfX_\ell(u), \mfX_j(u), I)$ is
    not an empty graph.
    
    Therefore, the result follows by applying Proposition~\ref{prop:biregular}
    with $A = S_{\alpha+1}^{(j)}$, $B = \mfX_j(u)$, $A' = S_{\alpha+1}^{(j)}(v) \subseteq
    S_{\alpha+1}^{(j)}$, and (hence) $N(A') = S_\delta^{(j)}(v)$.
\end{proof}

\begin{fact}\label{fact:large-rho}
    Let $1 \le j \le r-1$ be a color such that
    $\delta  = \dist_i(j) \ge 3$,  and $w$ be a vertex in $\mfX_j(u)$. 
     Let $1 \leq \alpha \le \delta-2$, and let $v$ be a vertex in  $S_\alpha^{(j)}$.
     If $\dist_i(v, w) = \delta - \alpha$, then 
     \[\{x: x \in \mfX_i(v) \text{ and } \dist_i(x, w) = \delta -  \alpha - 1 \} \subseteq S_{\alpha+1}^{(j)}(v).\]
 \end{fact}
 \begin{proof}
        For any $x \in \mfX_i(v)$, we have $\dist_i(u, x) \leq \alpha + 1$.  If
        $\dist_i(x, w) = \delta- \alpha - 1$, then $x \in
        \mfX_i^{(\alpha+1)}(u)$, because otherwise $\dist(u, w) < \delta$.
        Then $x$ is in $S_{\alpha + 1}^{(j)}(v)$, since there is a shortest
        from $u$ to $w$ passing through $x$.
 \end{proof}

\begin{proposition}\label{prop:large-slices}
    Let $1 \le j \le r-1$ be a color such that
    $\delta  = \dist_i(j) \ge 3$. Let $1 \le \alpha \le \delta-2$, and let $v \in
    S_\alpha^{(j)}$. Then
    \[
        |\mfX_i^{\delta-\alpha}(u)| \ge \frac{n_i|S_\delta^{(j)}(v)|}{|S_{\alpha+1}^{(j)}(v)|}.
    \]
\end{proposition}
\begin{proof}
    Let $k$ be a color satisfying $\dist_i(k) = \delta - \alpha$ and  $\mfX_k(v)\cap S_\delta^{(j)}(v) \ne \emptyset$. Let $w$ be a vertex in $ \mfX_k(v)\cap S_\delta^{(j)}(v)$. Consider the bipartite graph
    $B = (\mfX_i(v),\mfX_k(v), I)$, where $I = \{\ell : \dist_i(\ell) = \delta - \alpha - 1\}$. 

    By Fact \ref{fact:bipartite_regular}, $B$ is biregular, and
     by Fact~\ref{fact:large-rho} the degree of $w$ in $B$ is at most $|S_{\alpha+1}^{(j)}(v)|$. Denote by
    $d_k$ the degree of a vertex $x \in \mfX_i(v)$ in $B$, so $n_k
    |S_{\alpha+1}^{(j)}(v)| \ge n_i d_k$. Hence, summing over all colors $k$
    such that $\mfX_k(v)\cap S_\delta^{(j)}(v) \ne \emptyset$, we have
    \[
        |\mfX_i^{(\delta-\alpha)}(v)|
        \ge \sum_k n_k \ge \sum_k \frac{n_i
        d_k}{|S_{\alpha+1}^{(j)}(v)|} \ge \frac{n_i
        |S_\delta^{(j)}(v)|}{|S_{\alpha+1}^{(j)}(v)|}.
    \]
    Finally, by the coherence of $\mfX$, we have
    $|\mfX_i^{(\delta-\alpha)}(u)|= |\mfX_i^{(\delta-\alpha)}(v)|$.
\end{proof}

We now complete the proof of Lemma~\ref{lem:growth-spheres}.

\begin{proof}[Proof of Lemma~\ref{lem:growth-spheres}]
    Combining Corollary~\ref{cor:biregular-applied} and
    Proposition~\ref{prop:large-slices}, for any $1 \leq \alpha \le \delta - 2$
    we have
    \[
        |\mfX_i^{(\delta-\alpha)}(u)| \ge \frac{n_i n_k}{|S_{\alpha+1}^{(k)}|}
    \]
    and so since $S_{\alpha+1}^{(k)} \subseteq \mfX_i^{(\alpha+1)}(u)$ by
    definition, we have the desired inequality. 
\end{proof}


\section{Distinguishing number}\label{sec:dominant}

In this section, we will prove Lemma~\ref{lem:dominant}, which will allow us to
assume that our PCCs $\mfX$ satisfy $\rho = o(n^{2/3})$. 

We recall that the \emph{distinguishing number}
$D(i)$ of a color $i$ is the number of vertices $w$ such that $c(w,u) \ne
c(w,v)$, where $u$ and $v$ are any fixed pair of vertices such that $c(u,v) =
i$. Hence, $D(i) = \sum_{k \ne j} p^i_{jk^*}$. If $D(i)$ is
large for every color $i > 0$, then for every pair of distinct vertices $u,v
\in V$, a random individualized vertex $w$ gives different colors to $u$ and
$v$ in the stable refinement with good probability.  This idea is formalized in
the following lemma due to Babai~\cite{babai-cc}.

\begin{lemma}[{Babai~\cite[Lemma 5.4]{babai-cc}}]\label{lem:distinguishing}
    Let $\mfX$ be a PCC and let \newline $\zeta = \min\{D(i) : 1 \le i \le
    r-1\}$. Then there is a set of size $O(n\log n /\zeta)$ which completely
    splits $\mfX$.
\end{lemma}

We give the following lower bound on $\zeta$ when $\rho$ is sufficiently
large.

\begin{lemma}\label{lem:Di-large}
    Let $\mfX$ be a PCC and suppose that $\rho \ge  n^{2/3}(\log n)^{-1/3}$.
    \newline Then $D(i) =
    \Omega(n^{2/3}(\log n)^{-1/3})$ for all $1 \le i \le r-1$.
\end{lemma}

Lemma~\ref{lem:dominant} follows immediately from
Lemmas~\ref{lem:distinguishing} and~\ref{lem:Di-large}. \qed

We will prove Lemma~\ref{lem:Di-large} by separately addressing the cases $\rho
\ge n/3$ and $\rho < n/3$. The case $\rho < n/3$ will rely on our estimate for
the size of spheres in constituent digraphs, Lemma~\ref{lem:growth-spheres}.
For the case $\rho \ge n/3$, we will rely on Lemma~\ref{lem:few-large-colors},
which bounds the number of large colors when $D(i)$ is small for some color $i
\ge 1$. We prove Lemma~\ref{lem:few-large-colors} in the following subsection.

We first recall the following observation of Babai~\cite[Proposition
6.3]{babai-cc}.

\begin{proposition}[Babai]\label{prop:avg-Di}
    Let $\mfX$ be a PCC. Then
    \[
        \frac{1}{n-1}\sum_{j=1}^{r-1}D(j)n_j \ge \rho + 2.
    \]
\end{proposition}

The following corollary is then immediate. \qed

\begin{corollary}\label{cor:exists-large-Di}
    Let $\mfX$ be a PCC.  There exists a nondiagonal color $i$ with
    $D(i) > \rho$.
\end{corollary}

The following facts about the parameters of a coherent configuration are
standard.

\begin{proposition}[{\cite[Lemma 1.1.1, 1.1.2, 1.1.3]{zieschang-as}}]\label{prop:cc}
    Let $\mfX$ be a CC. Then for all colors $i,j,k$, the following relations
    hold:
    \begin{enumerate}
        \item[(i)] $n_i = n_{i^*}$
        \item[(ii)] $p^i_{jk} = p^{i^*}_{k^*j^*}$
        \item[(iii)] $n_i p^i_{jk} = n_jp^j_{ik^*}$
        \item[(iv)] $\sum_{j=0}^{r-1}p^i_{jk} = \sum_{j=0}^{r-1}p^i_{kj} = n_k$
    \end{enumerate}
\end{proposition}

\subsection{Bound on the number of large colors}

We now prove Lemma~\ref{lem:few-large-colors}, using the following preliminary
results.

\begin{lemma}\label{lem:few-small-colors}
    Let $\mfX$ be a PCC, let $I$ be a nonempty set of nondiagonal
    colors, let $n_I = \sum_{i \in I}n_i$, and let $J$ be the set of colors $j$
    such that $n_j \le n_I/2$.  Then 
    \[
        \sum_{j \in J}n_j \le 2\max\{D(i) : i \in I\}.
    \]
\end{lemma}
\begin{proof}
    For any color $i$, by Proposition~\ref{prop:cc}, we have 
    \begin{align*}
        D(i) &= \sum_{j=0}^{r-1}\sum_{k \ne j} p^i_{jk^*}
             = \sum_{j=0}^{r-1} \sum_{k \ne j} \frac{n_j p^j_{ik}}{n_i}\\
             &= \frac{1}{n_i} \sum_{j=0}^{r-1} n_j \sum_{k \ne j} p^j_{ik}
             = \frac{1}{n_i} \sum_{j=0}^{r-1} n_j(n_i - p^j_{ij}).
    \end{align*}

    Therefore,
    \begin{align*}
        n_I\max\{D(i) : i\in I\} &\ge \sum_{i\in I}
        n_i D(i)\\
        &\ge \sum_{i \in I}\sum_{j\in J}n_j(n_i - p^j_{ij})\\
        &\ge \sum_{j \in J} n_j \sum_{i \in I}(n_i - p^j_{ij})\\
        &\ge \sum_{j \in J} n_j \left(n_I - n_j\right)\\
        &\ge \frac{n_I}{2} \left(\sum_{j \in J} n_j\right).
    \end{align*}
\end{proof}

\begin{lemma}\label{lem:nondominant-smooth}
    Let $\mfX$ be a PCC, and suppose $p^i_{jk} > 0$ for some $i,j,k$.
    Then \[D(j) - D(k) \le D(i) \le D(j) + D(k).\]
\end{lemma}
\begin{proof}
    Fix vertices $u,v,w \in V$ with $c(u,w) = i$, $c(u,v)=j$, and $c(v,w)=k$.
    (These vertices exist since $p^i_{jk} > 0$.) For any vertex $x$ such that
    $c(x,u)\ne c(x,w)$, we have $c(x,u) \ne c(x,v)$ or $c(x,v) \ne c(x,w)$.
    Therefore, $D(j) + D(k) \ge D(i)$.
    
    For the other inequality, if $p^i_{jk} > 0$ then $p^j_{ik^*} > 0$ by
    Proposition~\ref{prop:cc}, and $D(k^*) = D(k)$ by the definition of
    distinguishing number. So we have $D(i) + D(k) = D(i) + D(k^*) \ge D(j)$,
    using the previous paragraph for the latter inequality.
\end{proof}

\begin{lemma}\label{lem:nondominant-smooth-appl}
    Let $\mfX$ be a PCC. Then for any nondiagonal color $i$ and number $0 \le
    \eta \le \rho - D(i)$, there is a color $j$ such that $\eta < D(j) \le \eta +
    D(i)$.
\end{lemma}
\begin{proof}
    By Corollary~\ref{cor:exists-large-Di}, there is a color $k$ with
    $D(k) > \rho$. Now consider a shortest path $u_0,\ldots, u_\ell$ in
    $\mfX_i$ with $c(u_0,u_\ell) = k$. (By the primitivity of $\mfX$, the
    digraph $\mfX_i$ is strongly connected, and such a path exists.) Let
    $\delta_j = D(c(u_0,u_j))$ for $1 \le j \le k$. By
    Lemma~\ref{lem:nondominant-smooth}, we have $|\delta_j - \delta_{j+1}| \le
    D(i)$. Hence, one of the numbers $\delta_j$ falls in the interval $(\eta,
    \eta + D(i)]$ for any $0 \le \eta \le \rho - D(i)$.
\end{proof}

We denote by $I_\alpha$ the set of colors $i$ with $D(i)\le \alpha$.

\begin{lemma}\label{lem:nondominant-basic}
    Let $\mfX$ be a PCC with $\rho > 0$. Let $i$ be a nondiagonal color and let $0 \le
    \eta \le \rho - 2D(i)$. Then
    \[
        n_i \le \sum_{j \in I_{\eta+3D(i)}\setminus I_\eta} n_j.
    \]
\end{lemma}
\begin{proof}
    By Lemma~\ref{lem:nondominant-smooth-appl}, the set
    $I_{\eta+2D(i)}\setminus I_{\eta+D(i)}$ is nonempty. Let $k \in
    I_{\eta+2D(i)}\setminus I_{\eta+D(i)}$. We have $\sum_{j=0}^{r-1}p^k_{ij} =
    n_i$ by Proposition~\ref{prop:cc}.  On the other hand, if $p^k_{ij} > 0$
    for some $j$, then $D(j)-D(i)\le D(k)\le D(j) + D(i)$ by
    Lemma~\ref{lem:nondominant-smooth}, and so $j \in I_{\eta+3D(i)}\setminus
    I_\eta$. Hence,
    \[
        n_i = \sum_{j=0}^{r-1}p^k_{ij}
        = \sum_{j \in I_{\eta+3D(i)}\setminus I_\eta} p^k_{ij}
        \le \sum_{j \in I_{\eta+3D(i)}\setminus I_\eta} n_j.
    \]
\end{proof}

\begin{lemma}\label{lem:nondominant-applied}
    Let $\mfX$ be a PCC with $\rho > 0$, let $i$ be a nondiagonal color, and let $0
    \le \eta \le \rho$. Then
    \[
        \left\lfloor\frac{\eta}{3D(i)}\right\rfloor n_i \le \sum_{j \in I_\eta} n_j.
    \]
\end{lemma}
\begin{proof}
    If $\eta < 3D(i)$, the left-hand side is $0$, so assume $\eta \ge 3D(i)$.
    For any integer $1 \le \alpha \le \lfloor \eta/(3D(i))\rfloor$, let
    $S_\alpha = I_{3D(i)\alpha} \backslash I_{3D(i)(\alpha-1)}$.  Then
    \[
        \bigcup_{\alpha = 1}^{ \lfloor \eta/(3D(i))\rfloor} S_\alpha \subseteq I_\eta
    \]
    By the disjointness of the sets $S_{\alpha}$ and
    Lemma~\ref{lem:nondominant-basic}, we have
    \[
        \sum_{j \in I_\eta} n_j \ge \sum_{\alpha = 1}^{\lfloor \eta/(3D(i))\rfloor}
        \sum_{j \in S_{\alpha}} n_j \ge \left\lfloor\frac{\eta}{3D(i)}\right\rfloor
        n_i.
    \]
\end{proof}

Finally, we are able to prove Lemma~\ref{lem:few-large-colors}.

\begin{proof}[Proof of Lemma~\ref{lem:few-large-colors}]
    Fix an integer $0 \le \alpha \le \lfloor \log_2(\rho/(3D(i)))\rfloor $.
    For any number $\beta$, let $J_\beta$ denote the set of colors $j$ such
    that $n_j \le \beta$.  We start by estimating $|J_{2^{\alpha}n_i}\setminus
    J_{2^{\alpha - 1} n_i}|$, i.e.,
    the number of colors $j$ with $2^{\alpha-1}n_i < n_j \le 2^\alpha n_i$. By
    Lemma~\ref{lem:nondominant-applied}, we have
    \[
        \sum_{j \in I_{2^\alpha(3D(i))}}n_j \ge 2^\alpha n_i.
    \]
    Therefore, applying Lemma~\ref{lem:few-small-colors} with $I =
    I_{2^{\alpha}(3D(i))}$ and $J = J_{2^{\alpha}n_i}$, we have
    \[
        \sum_{j \in J_{2^\alpha n_i}}n_j \le 2 \max\{D(i) : i\in I_{2^\alpha
        \cdot 3 D(i)}\} \le 2^{\alpha+1}(3D(i)),
    \]
    with the second inequality coming from the definition of
    $I_{2^{\alpha}(3D(i))}$.

    It follows that the number of colors $j$ such that $j \in J_{2^{\alpha}n_i}\setminus J_{2^{\alpha - 1} n_i}$
    is at most
    $2^{\alpha+1}(3D(i))/(2^{\alpha-1}n_i) = 12D(i)/n_i$.
    Overall, the number of colors $j$ satisfying 
    \[
        (1/2)n_i < n_j \le 2^{\lfloor\log_2(\rho/3D(i))\rfloor}n_i
    \]
    is at most $12(\log_2 n + 1)D(i)/n_i$.

    Furthermore, the number of colors $j$ satisfying
    \[
        n_j > 2^{\lfloor\log_2(\rho/3D(i))\rfloor}n_i \ge \frac{\rho n_i}{6D(i)}
    \]
    is at most $(6 D(i)/(\rho n_i))n$, since $\sum_{j=0}^{r-1}n_j = n$. Hence,
    the number of colors $j$ such that $n_j > n_i/2$ is at most $O((\log n +
    n/\rho)D(i)/n_i)$.
\end{proof}

\subsection{Estimates of the distinguishing number}

We now prove Lemma~\ref{lem:Di-large}, our lower bound for $D(i)$.

First, we recall the following two observations made by Babai~\cite[Proposition
6.4 and Theorem 6.11]{babai-cc}.

\begin{proposition}[Babai]\label{prop:babai-Di}
    Let $\mfX$ be a PCC. For colors $0 \le i,j \le r-1$, we have
    \[
        D(j) \le \dist_i(j)D(i).
    \]
\end{proposition}

\begin{proposition}[Babai]\label{prop:degree-Di}
    Let $\mfX$ be a PCC.  For any color $1 \le i \le r-1$, we have $n_i D(i)
    \ge n - 1$.
\end{proposition}

We prove the following two estimates of the distinguish number

\begin{lemma}\label{lem:Di-technical}
    Let $\mfX$ be a PCC. Fix nondiagonal colors $i,j \ge 1$ and a vertex $u \in
    V$. Let $\delta = \dist_i(j)$, and $\gamma =
    \sum_{\alpha=2}^{\delta-1}|\mfX_i^{(\alpha)}(u)|$. If $\delta \ge 3$, then
    \[
        D(i) = \Omega\left(\left(\frac{D(j)\sqrt{n
        n_j}}{\gamma}\right)^{2/3}\right).
    \]
\end{lemma}
\begin{proof}
    By Lemma~\ref{lem:growth-spheres}, for any $1 \leq \alpha \le \delta - 2$
    we have
    \[
        |\mfX_i^{(\alpha+1)}(u)||\mfX_i^{(\delta-\alpha)}(u)| \ge n_i n_j
    \]
    and in particular,
    \[
        \max\{|\mfX_i^{(\alpha+1)}(u)|, |\mfX_i^{(\delta-\alpha)}(u)|\} \ge
        \sqrt{n_i n_j}.
    \]
    Hence,
    \begin{equation}\label{eq:t}
        \gamma = \sum_{\alpha = 2}^{\delta-1} |\mfX_i^{(\alpha)}(u)| = \Omega(\delta\sqrt{n_i n_j}) =
        \Omega\left(\frac{\delta\sqrt{n n_j}}{\sqrt{D(i)}}\right),
    \end{equation}
    where the last inequality comes from Proposition~\ref{prop:degree-Di}.
    Now by Proposition~\ref{prop:babai-Di} and Eq.~\eqref{eq:t}, we have
    \[
        D(i) \ge \frac{D(j)}{\delta} = \Omega\left(\frac{D(j)\sqrt{n n_j}}{\gamma
        \sqrt{D(i)}}\right),
    \]
    from which the desired inequality immediately follows.
\end{proof}

\begin{lemma}\label{lem:nondominant-main}
    Let $\mfX$ be a PCC with $\rho = \Omega(n)$.  Then every nondiagonal color $i$ with $n_i
    \le \rho$ satisfies
    \[
        D(i) = \Omega\left(\sqrt{\frac{\rho n_i}{\log n}}\right).
    \]
\end{lemma}
\begin{proof}
    Fix a nondiagonal color $i$ with $n_i \le \rho$, and suppose $D(i) <
    \rho/6$ (otherwise the lemma holds trivially).  Let $J_\beta$ denote the
    set of colors $j$ such that $n_j \le \beta$.  Applying
    Lemma~\ref{lem:few-small-colors} with the set $I = \{i\}$, we have
    \begin{equation}\label{eq:J-small}
        \sum_{j \in J_{n_i/2}} n_j \le 2D(i).
    \end{equation}
    
    On the other hand, by Lemma~\ref{lem:nondominant-basic}, for every integer
    $\eta$ with $0 \le \eta \le \rho/2 - 3D(i)$, 
    \[
        n_i \le \sum_{j \in I_{\eta+3D(i)}\setminus I_\eta} n_j.
    \]
    Thus, for every such $\eta$, at least one of following two conditions hold:
    \begin{enumerate}
        \item[(a)] there exists a color $j \in I_{\eta+3D(i)}\setminus I_\eta$ satisfying
    $n_j > n_i / 2$; 
        \item[(b)] $\displaystyle{\sum_{\substack{j \in I_{\eta+3D(i)}\setminus I_\eta: \\ n_j \le n_i /
            2}} n_j \geq n_i}$.
    \end{enumerate}

    There are at least $\lfloor \rho/(6 D(i)) \rfloor$ disjoint sets of the
    form $I_{\eta+3D(i)}\setminus I_\eta$ with $0 \le \eta \le \rho/2 - 3D(i)$.
    By Lemma~\ref{lem:few-large-colors}, at most $O((\log n + n/\rho)D(i)/n_i)
    = O((\log n)D(i)/n_i)$ of these satisfy (a). By Eq.~\eqref{eq:J-small}, at
    most $2D(i)/n_i$ satisfy (b).  Hence, $\lfloor \rho/(6 D(i)) \rfloor =
    O((\log n)D(i)/n_i)$, giving the desired inequality.
\end{proof}

We recall that when color $1$ is dominant, it is symmetric. In this case, we
recall our notation $\mu = |N(x)\cap N(y)|$, where $x,y \in V$ are any pair of
vertices with $c(x,y)=1$ and $N(x)$ is the nondominant neighborhood of $x$.
Hence, $\mu = \sum_{i,j > 1}p^1_{ij}$. 

\begin{lemma}\label{lem:mu}
    Let $\mfX$ be a PCC with $n_1 \ge n/2$. Then $\mu \le \rho^2/n_1$.
\end{lemma}
\begin{proof}
    Fix a vertex $u$. There are at most $\rho^2$ paths of length two from $u$
    along edges of nondominant color, and exactly $n_1$ vertices $v$ such that
    $c(u,v) = 1$. For any such vertex $y$, there are exactly $\mu$ paths of
    length two from $u$ to $v$ along edges of nondominant color. Hence, $\mu
    \le \rho^2/n_1$.
\end{proof}

\begin{proof}[Proof of Lemma~\ref{lem:Di-large}]
    First, suppose $n^{2/3}(\log n)^{-1/3} \leq \rho < n/3$. We have $n_1 = n - \rho - 1
    > 2n / 3 - 1$. Consider two vertices $u,v \in V$ with $c(u,v)=1$. Note that
    for any vertex $w \in N(v)\setminus N(u)$, we have $c(w,u) = 1$ and $c(w,v)
    > 1$. Hence, by Lemma~\ref{lem:mu} and the definition of $D(1)$,
    \[
        D(1) \ge \rho - \mu \ge \rho - \frac{\rho^2}{n_1} \ge
        \left(\frac{1}{2} - o(1)\right)\rho = \Omega(n^{2/3}(\log n)^{-1/3}).
    \]
    Fix a color $i\ne 1$. If $\dist_i(1) = 2$, then  by Proposition~\ref{prop:babai-Di},
    \[
        D(i)\geq \frac{D(1)}{2} \ge \Omega(n^{2/3}(\log n)^{1/3}).
    \]
    Otherwise, if $\dist_i(1) \ge 3$, by applying Lemma~\ref{lem:Di-technical}
    with $j=1$, we have
    \[
        D(i) = \Omega\left(\left(\frac{D(1)\sqrt{n n_1}}{n - n_1}\right)^{2/3}\right) = \Omega\left(\left(\frac{\rho n}{\rho - 1}\right)^{2/3}\right) = \Omega(n^{2/3}).
    \]

    Now suppose $\rho \ge n/3$.  By Lemma~\ref{lem:nondominant-main} and
    Proposition~\ref{prop:degree-Di}, for every color $i$ with $n_i \le \rho$,
    we have
    \[
        (D(i))^{3/2} = \Omega\left(\sqrt{\frac{\rho n_i D(i)}{\log n}}\right) = \Omega\left(\sqrt{\frac{\rho n}{\log n}}\right),
    \]
    and hence $D(i) = \Omega(n^{2/3}(\log n)^{-1/3})$.
    If $n_1 \le \rho$, then
    $n_i \le \rho$ for all $i$, and we are done. Otherwise, if $n_1 > \rho$, we
    have only to verify that $D(1) = \Omega(n^{2/3}(\log n)^{-1/3})$.
    Consider two vertices $u, w$ with $\dist_1(u, w) = 2$. (Since we assume the
    rank is at least $3$, we can always find such $u, w$ by the primitivity of
    $\mfX$.) Let $i = c(u,w)$. Then $i > 1$ and so $n_i \le \rho$.  Since $D(i)
    = \Omega(n^{2/3}(\log n)^{-1/3})$ for every color $1 < i \leq r - 1$, and
    $\dist_1(i) = 2$, we have $D(1) = \Omega(n^{2/3}(\log n)^{-1/3})$ by
    Proposition~\ref{prop:babai-Di}.
\end{proof}


\section{Diameter of constituent graphs}\label{sec:diam2}

We now prove Lemma~\ref{lem:diam2}, which states that $\dist_i(1)=2$ for any
nondominant color $i$, assuming that inequality $\rho = o(n^{2/3})$.

We start from a few basic observations.

\begin{observation}\label{obs:ni-bound}
    Let $\mfX$ be a PCC.  For any nondominant color $i$, we have $n_i
    \ge n_1/\rho$.
\end{observation}
\begin{proof}
    Fix a vertex $u \in V$. Since $\mfX_i$ is connected ($\mfX$ is primitive), for any $v \in
    \mfX_1(x)$, there is a shortest path in $\mfX_i$ from $u$ to $v$, hence there exists a
    vertex $w \in N(u)$ such that $c(w,v) = i$, so $|N(x)|n_i \ge |\mfX_1(x)|$.
\end{proof}

\begin{lemma}\label{lem:twodistance-basic}
Let $\mfX$ be a PCC with a nondominant color $i$, let $\delta = \dist_1(i)$,
and suppose $\delta \geq 3$.  Any vertices $u,w$ with $c(u,w)=i$ satisfy the
following two properties:
\begin{enumerate}
\item[(i)] If $v \in \mfX_i^{(\delta -1)}(u) \cap \mfX_i^{(\delta - 2)}(w)$, then
\[\mfX_{i}(v) \cap  \mfX_1(u) \subseteq \mfX_{i}(v) \cap  \mfX_i^{(\delta - 1)}(w);\]
\item[(ii)] If $z \in \mfX_1(u) \cap \mfX_i^{(\delta - 1)}(w)$, then
\[\mfX_{i^*}(z) \cap  \mfX_i^{(\delta - 2)}(w) \subseteq \mfX_{i^*}(z) \cap  \mfX_i^{(\delta - 1)}(u).\]
\end{enumerate}
\end{lemma}
\begin{proof}
If $v \in \mfX_i^{(\delta -1)}(u) \cap \mfX_i^{(\delta - 2)}(w)$ then $\dist_i(u, v) = \delta - 1$ and 
$\dist_i(w, v) = \delta - 2$.
So for any vertex $x \in \mfX_{i}(v)$, we have $\dist_i(w, x) \leq \delta - 1$.
If $x \in \mfX_1(u)$, then $\dist_i(w, x) = \delta - 1$, since otherwise $\dist_i(u, x) < \delta$.

Similarly, $z \in \mfX_1(u) \cap \mfX_i^{(\delta - 1)}(w)$ means $\dist_i(u, z) = \delta$ and 
$\dist_i(w, z) = \delta - 1$.
So for any $y$ satisfying $\dist_i(w, y) = \delta - 2$,
we have $\dist_i(u, y) \leq \delta - 1$. 
If $z \in \mfX_i(y)$, then $\dist_i(u, y) = \delta - 1$, since otherwise $\dist_i(u, z) < \delta$.
\end{proof}

\begin{lemma}\label{lem:betabound}
Let $\mfX$ be a PCC with a nondominant color $i$,
let $\delta = \dist_1(i)$, and suppose $\delta \ge
3$. Fix a vertex $u \in V$ and let $B$ be the bipartite graph
$(\mfX_i^{(\delta-1)}(u), \mfX_1(u), i)$. Let $\gamma$ denote the minimum degree in $B$ of a
vertex in $\mfX_1(u)$, and let $\beta$ denote the maximum degree in $B$ of
a vertex in $\mfX_i^{(\delta-1)}(u)$. Then $\beta \le \gamma\rho^2/n_1$.
\end{lemma}
\begin{proof}
In fact, $B$ is regular on $\mfX_1(u)$ by Fact~\ref{fact:bipartite_regular}, so
every vertex in $\mfX_1(u)$ has degree $\gamma$ in $B$.

Let $v \in \mfX_i^{(\delta-1)}(u)$ achieve degree $\beta$ in $B$, and let $w
\in \mfX_i(u)$ be such that $\dist_i(w, v) = \delta - 2$.  Let $B'$ be the subgraph
of $B$ given by $(S_1, S_2, i)$, where $S_1 = \mfX_i^{(\delta -1)}(u) \cap
\mfX_i^{(\delta - 2)}(w)$ and $S_2 = \mfX_1(u) \cap \mfX_i^{(\delta - 1)}(w)$.
Note that $v \in S_1$. By Lemma~\ref{lem:twodistance-basic}~(i), every neighbor
of $v$ in $B$ is also a neighbor of $v$ in $B'$, and in particular, the degree
of $v$ in $B'$ is again $\beta$.

Let $j = c(w, v)$, and let $H = (\mfX_j(w), S_3, i)$, where
\[
S_3 = \left\{z \in \mfX_i^{(\delta-1)}(w): |\mfX_{i^*}(z)\cap \mfX_j(w)| \le \gamma\right\}.
\]
Recall that $v \in \mfX_j(w)$, so $v$ is also a vertex of $H$. We claim that
every neighbor of $v$ in $B'$ is also a neighbor of $v$ in $H$, so the
degree of $v$ in $B'$ is again $\ge \beta$. Indeed, let $z \in \mfX_i(v) \cap
S_2$. Then $z \in \mfX_i^{(\delta-1)}(w)$, and furthermore
\[
|\mfX_{i^*}(z) \cap  \mfX_j(w)| \le |\mfX_{i^*}(z) \cap  \mfX_i^{(\delta - 2)}(w)| 
\le |\mfX_{i^*}(z) \cap  \mfX_i^{(\delta - 1)}(u)| = \gamma.
\]
So, $z \in S_3$, and every neighbor of $v$ in $B'$ is also a vertex of $H$ as
claimed.

Now by Fact~\ref{fact:bipartite_regular}, $H$ is regular on $\mfX_j(w)$ with
degree $\ge \beta$.  Hence,
\[
    \beta n_j \leq |E(H)| \leq \gamma |S_3| \leq  \gamma |\mfX_i^{(\delta-1)}(w)| \leq \gamma \rho.
\]
The lemma follows by Observation~\ref{obs:ni-bound}.
\end{proof}

\begin{proof}[Proof of Lemma~\ref{lem:diam2}]
    We prove that if $\dist_i(1) \ge 3$ for some color $i$, then $\rho
    \gtrsim n^{2/3}$. Without loss of generality, we assume $n_1 \sim n$, since
    otherwise we are already done.

    Let $i$ be such that $\dist_i(1) \ge 3$,
    and write $\delta = \dist_i(1)$.
    Fix a vertex $u$ and let $B$, $\gamma$, and $\beta$ be as in
    Lemma~\ref{lem:betabound}, so $\beta \le \gamma\rho^2/n_1$.
    By Lemma~\ref{lem:betabound}, $B$ is regular on $\mfX_1(u)$. 
    Let $\gamma, \beta$ be defined as Lemma~\ref{lem:betabound}.
    By Lemma~\ref{lem:betabound}, we have $\beta \leq \gamma\rho^2/n_1$.
    Therefore, by counting the number of edges in $B$
    \[
        \frac{\rho^3 \gamma}{n_1} \ge \beta|\mfX_i^{(\delta-1)}(x)| \ge |E(B)| = n_1\gamma.
    \]
    The lemma is then immediate since $n_1 \sim n$.
\end{proof}


\section{Clique geometries}\label{sec:clique}

In this section, we prove Theorem~\ref{thm:clique-geometry}, giving
sufficient conditions for the existence of an asymptotically uniform clique
geometry in a PCC. 

We use the word ``geometry'' in Definition~\ref{def:clique-geom} because the
cliques resemble lines in a geometry: two distinct cliques intersect in
at most one vertex. Indeed, a regular graph $G$ has a clique geometry $\mcG$
with cliques of uniform order only if it is the point-graph of a geometric
1-design with lines corresponding to cliques of $\mcG$.


Theorem~\ref{thm:clique-geometry} builds on earlier work of
Metsch~\cite{metsch-clique} on the existence of similar clique structures in
``sub-amply regular graphs'' (cf.~\cite{bw-delsarte}) via the following lemma.
The lemma can be derived from~\cite[Theorem 1.2]{metsch-clique}, but
see~\cite[Lemma 4]{bw-delsarte} for a self-contained proof.
\begin{lemma}\label{lem:metsch-local}
    Let $H$ be a graph on $k$ vertices which is regular of degree $\lambda$
    and such that any pair of nonadjacent vertices have at most $\mu$ common
    neighbors. Suppose that $k\mu = o(\lambda^2)$. Then there is a
    partition of $V(H)$ into maximal cliques of order $\sim \lambda$, and all
    other maximal cliques of $H$ have order $o(\lambda)$.
\end{lemma}

Metsch's result, applied to the graphs induced by $G(\mfX)$ on sets of the form
$\mfX_i(u)$, gives collections of cliques which locally resemble asymptotically
uniform clique geometries. These collections satisfy the following definition
for a set $I = \{i\}$ containing a single color.

\begin{definition}\label{def:local-clique}
    Let $I$ be a set of nondominant colors. An \emph{$I$-local clique
    partition} at a vertex $u$ is a collection $\mcP$ of subsets of $\mfX_I(u)$
    satisfying the following properties:
    \begin{enumerate}
        \item[(i)] $\mcP$ is a partition of $\mfX_I(u)$ into maximal cliques in the
            subgraph of $G(\mfX)$ induced on $\mfX_I(u)$;
        \item[(ii)] for every $C \in \mcP_u$ and $i \in I$, we have $|C\cap
            \mfX_i(u)|\sim \lambda_i$.
    \end{enumerate}
    We say $\mfX$ \emph{has $I$-local clique partitions} if there is an
    $I$-local clique partition at every vertex $u \in V$.
\end{definition}

To prove Theorem~\ref{thm:clique-geometry}, we will stitch local clique
partitions together into geometric clique structures.

Note that from the definition, if $\mcP$ is an $I$-local clique partition (at
some vertex) and $i \in I$, then $|\mcP| \sim n_i/\lambda_i$.
\begin{corollary}\label{cor:single-color-clique}
    Let $\mfX$ be a PCC and let $i$ be a nondominant color such that
    $n_i \mu = o(\lambda_i^2)$. Then $\mfX$ has $\{i\}$-local clique
    partitions.
\end{corollary}
\begin{proof}
    Fix a vertex $u$, and apply Lemma~\ref{lem:metsch-local} to the graph $H$
    induced by $G(\mfX)$ on $\mfX_i(u)$. The Lemma gives a collection of
    cliques satisfying Definition~\ref{def:local-clique}.
\end{proof}

The following simple observation is essential for the proofs of this section.

\begin{observation}\label{obs:clique-basic}
    Let $\mfX$ be a PCC, let $C$ be a clique in $G(\mfX)$, and suppose $u \in V\setminus C$ is such that
    $|N(u)\cap C| > \mu$. Then $C \subseteq N(u)$.
\end{observation}
\begin{proof}
    Suppose there exists a vertex $v \in C\setminus N(u)$, so $c(u,v) = 1$. Then $|N(u)
    \cap N(v)| = \mu$ by the definition of $\mu$ in a PCC. But
    \begin{align*}
        |N(u)\cap N(v)| &\ge |N(u) \cap C \cap N(v)| = |N(u) \cap (C \backslash
        \{v\})| \\
        &= |N(u) \cap C| > \mu,
    \end{align*}
    a contradiction.
\end{proof}

Under modest assumptions, if local clique partitions exist, they are unique.

\begin{lemma}\label{lem:lcp-unique}
    Let $\mfX$ be a PCC, let $i$ be a nondominant color such that $n_i\mu =
    o(\lambda_i^2)$, and let $I$ be a set of nondominant colors such that $i
    \in I$. Suppose $\mfX$ has $I$-local clique partitions. Then for every
    vertex $u \in V$, there is a unique $I$-local clique partition $\mcP$ at
    $u$.
\end{lemma}
\begin{proof}
    Let $u \in V$ and let $\mcP$ be an $I$-local clique partition at $u$.
    Let $C$ and $C'$ be two distinct maximal cliques in the subgraph 
    of $G(\mfX)$ induced on $\mfX_I(u)$. We show that $|C \cap C'| < \mu$.
    Suppose for the contradiction that $|C\cap C'| \ge \mu$.
    For a vertex $v \in C \setminus C'$, we have $|N(v)\cap
    (C'\cup\{u\})| > \mu$, and so $C' \subseteq N(v)$ by
    Observation~\ref{obs:clique-basic}. But since $v \notin C'$, this contradicts 
    the maximality of $C'$. So in fact $|C \cap C'| < \mu$.
    
    Now let $C \notin \mcP$ be a maximal clique in the subgraph of $G(\mfX)$ 
    induced on $\mfX_I(u)$. Since $\mcP$ is an $I$-local clique partition,
    it follows that
    \[
        |C| = \sum_{C' \in \mcP} |C' \cap C| < \mu |\mcP|
        \sim n_i\mu/\lambda_i = o(\lambda_i).
    \]
    Then $C$ does not belong to an $I$-local clique partition, since it fails
    to satisfy Definition~\ref{def:local-clique}~(ii).
\end{proof}

\subsection{Local cliques and symmetry}

Suppose $\mfX$ has $I$-local clique partitions, and $c(u,v) \in I$ for some
$u,v \in V$. We remark that in general, the clique containing $v$ in the
$I$-local clique partition at $u$ will not be in any way related to any clique
in the $I$-local clique partition at $v$. In particular, we need not have
$c(v,u) \in I$. However, even when $c(v,u) \in I$ as well, there is no
guarantee that the clique at $u$ containing $v$ will have any particular
relation to the clique at $v$ containing $u$. This lack of symmetry is a
fundamental obstacle that we must overcome to prove
Theorem~\ref{thm:clique-geometry}.

Lemma~\ref{lem:clique-eq-sym} below is the main result of this subsection. It
gives sufficient conditions on the parameters of a PCC for finding the desired
symmetry in local clique partitions satisfying the following additional
condition.

\begin{definition}\label{def:str-local-clique}
    Let $I$ be a set of nondominant colors, let $u \in V$, and let $\mcP$ be an
    $I$-local clique partition at $u$. We say $\mcP$ is \emph{strong} if for
    every $C \in \mcP$, the clique $C \cup \{u\}$ is maximal in $G(\mfX)$. We
    say $\mfX$ has \emph{strong} $I$-local clique partitions if there is a
    strong $I$-local clique partition at every vertex $u \in V$.
\end{definition}

We introduce additional notation. Suppose $I$ is a set of nondominant colors,
and $i \in I$ satisfies $n_i\mu = o(\lambda_i)^2$. If $\mfX$ has $I$-local clique
partitions, then for every $u,v \in V$ with $c(u,v) \in I$, we denote by
$K_I(u,v)$ the set $C\cup \{u\}$, where $C$ is the clique in the partition of
$\mfX_I(u)$ containing $v$ (noting that by Lemma~\ref{lem:lcp-unique}, this
clique is uniquely determined).  

\begin{lemma}\label{lem:clique-eq-sym}
    Let $\mfX$ be a PCC with $\rho = o(n^{2/3})$, let $i$ be a
    nondominant color, and let $I$ and $J$ be sets of nondominant colors such
    that $i \in I$, $i^* \in J$, and $\mfX$ has strong $I$-local and $J$-local
    clique partitions.  Suppose $\lambda_i\lambda_{i^*} = \Omega(n)$.
    Then for every $u,v \in V$ with $c(u,v) = i$, we have $K_I(u,v) =
    K_J(v,u)$.
\end{lemma}

We first prove two easy preliminary statements.

\begin{proposition}\label{prop:23}
    Suppose $\rho = o(n^{2/3})$. Then $\mu = o(n^{1/3})$ and $\mu\rho = o(n)$.
    Furthermore, for every nondominant color $i$, we have $\mu = o(n_i)$.
\end{proposition}
\begin{proof}
   By Lemma~\ref{lem:mu}, $\mu \leq \rho^2 / n_1 = o(n^{1/3})$, and then $\mu
   \rho = o(n)$. The last inequality follows by Lemma~\ref{lem:diam2}.
\end{proof}

\begin{lemma}\label{lem:clique-overlap-eq}
    Let $\mfX$ be a PCC and let $I$ and $J$ be sets of nondominant colors such
    that $\mfX$ has strong $I$-local and $J$-local clique partitions.  Suppose
    that for some vertices $u,v,x,y \in V$ we have $|K_I(u,v)\cap K_J(x,y)| >
    \mu$. Then $K_I(u,v) = K_J(x,y)$.
\end{lemma}
\begin{proof}
    Suppose there exists a vertex
    $z \in K_J(x, y) \setminus K_I(u, v)$.
    We have $|N(z)\cap K_I(u,v)| \geq |K_J(x,y) \cap K_I(u,v)|  > \mu$. Then
    $K_I(u,v) \subseteq N(z)$ by Observation~\ref{obs:clique-basic},
    contradicting the maximality of $K_I(u,v)$.  Thus, $K_J(x,y) \subseteq
    K_I(u, v)$.  Similarly, $K_I(u,v)\subseteq K_J(x,y)$.
\end{proof}

\begin{proof}[Proof of Lemma~\ref{lem:clique-eq-sym}]
    Without loss of generality, assume $\lambda_i \le \lambda_{i^*}$.

    Suppose for contradiction that there exists a vertex $u \in V$ such that for \emph{every} $v
    \in \mfX_i(u)$, we have $K_I(u,v) \ne K_J(v,u)$.  Then $|K_I(u,v) \cap K_J(v,u)| \le \mu$ by
    Lemma~\ref{lem:clique-overlap-eq}. Fix $v \in \mfX_i(u)$, so for every $w
    \in K_I(u,v)\cap \mfX_i(u)$, we have $|K_J(w,u)\cap K_I(u,v)| \le \mu$.
    Hence, there exists some sequence $w_1,\ldots,w_\ell$ of $\ell =
    \lceil\lambda_i/(2\mu)\rceil$ vertices $w_\alpha \in K_I(u,v)\cap
    \mfX_i(u)$ such that $K_J(w_\alpha,u) \ne K_J(w_\beta,u)$ for $\alpha \ne
    \beta$. But by Lemma~\ref{lem:clique-overlap-eq}, for
    $\alpha \ne \beta$ we have $|K_J(w_\alpha,u)\cap K_J(w_\beta,u)| \le \mu$. Hence, for any $1
    \le \alpha \le \ell$ we have
    \[
        \left|K_J(w_\alpha,u)\setminus \bigcup_{\beta \ne \alpha} K_J(w_\beta,u)\right| \gtrsim
        \lambda_{i^*} - \mu\lambda_i/(2\mu) \ge \lambda_{i^*}/2.
    \]
    But $K_J(w_\alpha,u) \subseteq N(u)$, so
    \[
        |N(u)| \ge \left|\bigcup_{\alpha=1}^{\ell} K_J(w_\alpha,u)\right| \gtrsim
        \frac{\lambda_i\lambda_{i^*}}{4\mu} = \omega(\rho)
    \]
    by Proposition~\ref{prop:23}. This contradicts the definition of
    $\rho$.

    Hence, for any vertex $u$, there is some $v \in \mfX_i(u)$ such that
    $K_I(u,v) = K_J(v,u)$.  Then, in particular, $|\mfX_{i^*}(v) \cap
    \mfX_I(u)| \gtrsim \lambda_{i^*}$ by the definition of a $J$-local clique
    partition. By the coherence of $\mfX$, for every $v \in \mfX_i(u)$, we have
    $|\mfX_{i^*}(v)\cap \mfX_I(u)| \gtrsim \lambda_{i^*}$. Recall that
    $\mfX_I(u)$ is partitioned into $\sim n_i/\lambda_i$ maximal cliques, and
    for each of these cliques $C$ other than $K_I(u,v)$, we have $|N(v)\cap C|
    \le \mu$.  Hence,
    \[
        |\mfX_{i^*}(v)\cap K_I(u,v)| \gtrsim \lambda_{i^*} - O\left(\frac{\mu
        n_i}{\lambda_i}\right) = \lambda_{i^*} -
        o\left(\frac{n}{\lambda_i}\right) \sim \lambda_{i^*}
    \]
    by Proposition~\ref{prop:23}. Since the $J$-local clique partition at $v$
    partitions $\mfX_{i^*}(v)$ into $\sim n_i/\lambda_{i^*}$ cliques, at least
    one of these intersects $K_I(u,v)$ in at least $\sim \lambda_{i^*}^2/n_i =
    \omega(\mu)$ vertices. In other words, there is some $x \in \mfX_{i^*}(v)$
    such that $|K_J(v,x) \cap K_I(u,v)| = \omega(\mu)$.  But then $K_J(v,x) =
    K_I(u,v)$ by Lemma~\ref{lem:clique-overlap-eq}. In particular, $u \in
    K_J(v,x)$, so $K_J(v,x) = K_J(v,u)$. Hence, $K_J(v,u) = K_J(v,x) =
    K_I(u,v)$, as desired.
\end{proof}

\subsection{Existence of strong local clique partitions}

Our next step in proving Theorem~\ref{thm:clique-geometry} is showing the
existence of strong local clique partitions. We accomplish this via the
following lemma.

\begin{lemma}\label{lem:local-clique-strong}
    Let $\mfX$ be a PCC such that $\rho = o(n^{2/3})$, and let $i$ be
    a nondominant color such that $n_i\mu = o(\lambda_i^2)$. Suppose that for
    every color $j$ with $n_j < n_i$, we have $\lambda_j = \Omega(\sqrt{n})$.
    Then for $n$ sufficiently large, there is a set $I$ of nondominant colors
    with $i \in I$ such that $\mfX$ has strong $I$-local clique partitions.
\end{lemma}

We will prove Lemma~\ref{lem:local-clique-strong} via a sequence of lemmas
which gradually improve our guarantees about the number of edges between
cliques of the $I$-local clique partition at a vertex $u$ and the various
neighborhoods $\mfX_j(u)$ for $j \notin I$.

\begin{lemma}\label{lem:triple-count}
    Let $\mfX$ be a PCC, and let $i$ and $j$ be nondominant colors.
    Then for any $0< \eps <1$ and any $u,v \in V$ with $c(u,v) = j$, we have
    \[
        |\mfX_i(u)\cap N(v)| \le \max \left\{\frac{\lambda_i+1}{1 - \eps},
        n_i\sqrt{\frac{\mu}{\eps n_j}}\right\}
    \]
\end{lemma}
\begin{proof}
    Fix $u,v \in V$ with $c(u,v) = j$ and let $\alpha = |\mfX_i(u)\cap N(v)|$. We
    count the number of triples $(x,y,z)$ of vertices such that $x,y \in
    \mfX_i(u)\cap N(z)$, with $c(u,z)=j$ and $c(x,y)=1$. There
    are at most $n_i^2$ pairs $x,y \in \mfX_i(u)$, and if $c(x,y)=1$ then there
    are at most $\mu$ vertices $z$ such that $x,y \in N(z)$. Hence, the number
    of such triples is at most $n_i^2\mu$. On the other hand, by the coherence
    of $\mfX$, for every $z$ with $c(u,z) = j$, we have at least
    $\alpha(\alpha-\lambda_i-1)$ pairs $x,y \in \mfX_i(u)\cap N(z)$ with
    $c(x,y)=1$. Hence, there are at least $n_j \alpha(\alpha-\lambda_i-1)$
    total such triples.  Thus, 
    \[
        n_j\alpha(\alpha-\lambda_i-1) \le n_i^2\mu.
    \]
    Hence, if $\alpha \leq (\lambda_i + 1) / (1 - \eps)$, then we are done.
    Otherwise, $\alpha > (\lambda_i + 1) / (1 - \eps)$, and then $\lambda_i + 1
    < (1 - \eps)\alpha$. So, we have
    \[
        n_i ^2 \mu \ge n_j \alpha (\alpha - \lambda_i - 1) > \eps n_j \alpha^2,
    \]
    and then $\alpha < n_i \sqrt{\mu/(\eps n_j)}$.
\end{proof}

\begin{lemma}\label{lem:local-clique-pre}
    Let $\mfX$ be a PCC, and let $i$ be a nondominant color such that $n_i\mu =
    o(\lambda_i^2)$.  Let $I$ be a set of nondominant colors with $i \in I$
    such that $\mfX$ has $I$-local clique partitions.  Let $j$ be a nondominant
    color such that $n_i\sqrt{\mu/n_j} < (\sqrt{3}/2)\lambda_i$.   
    Let $u \in V$, let $\mcP_u$ be the $I$-local clique partition at $u$, and let $v
    \in \mfX_j(u)$. Suppose some clique $C \in \mcP_u$ is such that $c(u,v) = j$
    and $|N(v)\cap C| \ge \mu$. Then for every vertex $x,y \in V$ with $c(x,y)
    = j$, letting $\mcP_x$ be the $I$-local clique partition at $x$, the
    following statements hold: 
    \begin{enumerate}
        \item[(i)] there is a unique clique $C \in \mcP_x$ such that $C
            \subseteq N(y)$;
        \item[(ii)] $|N(y)\cap \mfX_i(x)| \sim \lambda_i$.
    \end{enumerate}
\end{lemma}
\begin{proof}
    Letting $\widehat{C} = C \cup \{u\}$, we have $|\widehat{C}\cap N(v)| \ge \mu+1 > \mu$.
    Therefore, by Observation~\ref{obs:clique-basic},
    we have $C \subseteq N(v)$. In particular, $|N(v)\cap \mfX_i(u)| \gtrsim
    \lambda_i$, and so by the coherence of $\mfX$, $|N(y)\cap \mfX_i(x)|
    \gtrsim \lambda_i$ for every pair $x,y \in V$ with $c(x,y) = j$.

    Now fix $x \in V$, and let $\mcP_x$ be the $I$-local clique partition at
    $x$.  By the definition of an $I$-local clique partition, we have
    $|\mcP_x| \sim n_i/\lambda_i$. For every $y \in \mfX_j(x)$, by assumption we have  
    \begin{equation}\label{eq:xy-i-nbrhd-lower}
        |N(y)\cap \mfX_i(x)|\gtrsim \lambda_i = \omega(\mu n_i/\lambda_i).
    \end{equation}
    Then it follows from the pigeonhole principle that for $n$ sufficiently
    large, there is some clique $C \in \mcP_x$ such that $|N(y)\cap C| > \mu$,
    and then $C \subseteq N(y)$ by Observation~\ref{obs:clique-basic}. 

    Now suppose for contradiction that there is some clique $C' \in \mcP_x$
    with $C' \ne C$, such that $C' \subseteq N(y)$.
    \begin{equation}\label{eq:localclique}
        |N(y)\cap \mfX_i(x)| \ge |C \cup C'| \gtrsim 2\lambda_i \sim
        2(\lambda_i + 1)
    \end{equation}
    (with the last relation holding since $\lambda_i = \omega(\sqrt{n_i\mu}) =
    \omega(1)$.) However, by Lemma~\ref{lem:triple-count} with $\eps = 1/3$, we
    have
    \[
        |\mfX_i(x)\cap N(y)| \le
        \max\left\{\frac{3}{2}(\lambda_i+1),n_i\sqrt{\frac{3\mu}{n_j}}\right\}
        =\frac{3}{2}(\lambda_i + 1),
    \]
    with the last equality holding by assumption. This contradicts
    Eq.~\eqref{eq:localclique}, so we conclude that $C$ is the unique clique in
    $\mcP_x$ satisfying $C \subseteq N(y)$. In particular, by
    Observation~\ref{obs:clique-basic}, we have $|N(y) \cap C'| \le \mu$ for
    every $C' \in \mcP_x$ with $C' \ne C$.

    Finally, we estimate $|N(y)\cap \mfX_i(x)|$ by
    \begin{align*}
        |N(y)\cap \mfX_i(x) \cap C| &+ \sum_{C' \ne
        C}|N(y)\cap \mfX_i(x) \cap C'| \\
        \lesssim \lambda_i &+ \mu n_i/\lambda_i
        \sim \lambda_i,
    \end{align*}
    which, combined with Eq.~\eqref{eq:xy-i-nbrhd-lower}, gives $|N(y)\cap
    \mfX_i(x)| \sim \lambda_i$.
\end{proof}

\begin{lemma}\label{lem:local-clique-exists}
    Let $\mfX$ be a PCC, and let $i$ be a nondominant color such that
    $n_i\mu = o(\lambda_i^2)$. There exists a set $I$ of nondominant colors
    with $i \in I$ such that $\mfX$ has $I$-local clique partitions and the
    following statement holds. Suppose $j$ is a nondominant color such that
    $n_i\sqrt{\mu/n_j} = o(\lambda_i)$, let $u \in V$, and let $\mcP$ be the 
    $I$-local clique partition at $u$. Then for any $C \in \mcP$ and any
    vertex $v \in \mfX_j(u)\setminus C$, we have $|N(v)\cap C| < \mu$.
\end{lemma}
\begin{proof}
    By Corollary~\ref{cor:single-color-clique}, $\mfX$ has $\{i\}$-local clique
    partitions. Let $I$ be a maximal subset of of the nondominant
    colors such that $i \in I$ and $\mfX$ has $I$-local clique partitions.
    We claim that $I$ has the desired property.

    Indeed, suppose there exists some color $j \notin I$ satisfying $n_i\sqrt{\mu/n_j} =
    o(\lambda_i)$, some vertices $u,v$ with $c(u,v) = j$, and some $C \in
    \mcP$ with $|N(v)\cap C| \ge \mu$, where $\mcP$ is the $I$-local clique
    partition at $u$. By Lemma~\ref{lem:local-clique-pre}, for $n$
    sufficiently large, for every vertex $u,v \in V$ with $c(u,v) = j$, and
    $I$-local clique partition $\mcP$ at $u$, there is a unique clique $C \in
    \mcP$ such that $C \subseteq N(v)$, and furthermore
    \begin{equation}\label{eq:intersection}
    |N(v) \cap \mfX_i(u)|  \sim \lambda_i\,.
    \end{equation}

    Now fix $u \in V$ and let $\mcP$ be the $I$-local clique partition at $u$.
    Let $\mcP'$ be the collection of sets $C'$ of the form
    \[
        C' = C \cup \{v \in \mfX_j(u) : C \subseteq N(v)\}
    \]
    for every $C \in \mcP$. Let $J = I\cup \{j\}$. We claim that $\mcP'$
    satisfies Definition~\ref{def:local-clique}, so
    $\mfX$ has local clique partitions on $J$. This contradicts the maximality
    of $I$, and the lemma then follows.
    
    First we verify Definition~\ref{def:local-clique}~(i).  By
    the second paragraph of this proof, $\mcP'$ partitions $\mfX_J(u)$. Furthermore,
    the sets $C \in \mcP'$ are cliques in $G(\mfX)$, since for any $C \in
    \mcP'$ and any distinct $v,w \in C\cap \mfX_j(u)$, we have
    \[
        |N(v)\cap N(w)| \ge |C \cap \mfX_I(u)| \gtrsim \lambda_i = \omega(\mu)
    \]
    and so $c(v,w)$ is nondominant by the definition of $\mu$. Furthermore, the
    cliques $C \in \mcP'$ are maximal in the subgraph of G(\mfX) induced on
    $\mfX_J(u)$, since they are maximal when restricted to $\mfX_I(u)$ and each
    vertex $v \in \mfX_j(u)$ extends a unique clique in the restriction of
    $\mcP'$ to $\mfX_I(u)$.

    We now verify Definition~\ref{def:local-clique}~(ii). By the
    pigeonhole principle, there is some $C \in \mcP'$ with 
    \[
        |C\cap \mfX_j(u)|\gtrsim \frac{n_j}{|\mcP'|} = \frac{n_j}{|\mcP|}  \sim \frac{\lambda_i
        n_j}{n_i}.
    \]
    But since $C$ is a clique in $G(\mfX)$, we have $|C\cap \mfX_j(u)| \le
    \lambda_j + 1$. So, from the defining property of $j$,
    \[
        \lambda_j + 1\gtrsim \frac{\lambda_i n_j}{n_i} = \omega(\sqrt{\mu n_j})
    \]
    Since $n_j$ and $\mu$ are positive integers, we have in particular $\lambda_j = \omega(1)$, and
    thus
    \begin{equation}\label{eq:lambda-ij}
        \lambda_j \gtrsim \frac{\lambda_i n_j}{n_i} = \omega(\sqrt{\mu n_j}).
    \end{equation}
    Hence, $n_j \mu = o(\lambda_j^2)$, and so by
    Corollary~\ref{cor:single-color-clique}, $\mfX$ has $\{j\}$-local clique
    partitions.
    
    Let $C' \subseteq \mfX_j(u)$ be a maximal clique in $G(\mfX)$ of order
    $\sim \lambda_j$. By Eq.~\eqref{eq:intersection} there are $\sim
    \lambda_j\lambda_i$ nondominant edges between $C'$ and $\mfX_i(u)$, so some
    $x \in \mfX_i(u)$ satisfies
    \[
        |N(x)\cap C'| \gtrsim \lambda_j\lambda_i/n_i =
        \omega(\lambda_j\sqrt{\mu/n_j}) = \omega(\mu).
    \] 
    (The last equality uses Eq.~\eqref{eq:lambda-ij}.)
    Furthermore, by
    Eq.~\eqref{eq:lambda-ij}, we have
    \[
        n_j \lesssim (n_i/\lambda_i)\lambda_j =
    o(\sqrt{n_i/\mu}\lambda_j),
    \]
    where the last inequality comes from the assumption that $\sqrt{n_i\mu} =
    o(\lambda_i)$.  So by applying Lemma~\ref{lem:local-clique-pre} with
    $\{j\}$ in place of $I$, it follows that for every $x \in \mfX_i(u)$, we
    have $|N(x)\cap \mfX_j(u)| \sim \lambda_j$.

    We count the nondominant edges between $\mfX_i(u)$ and $\mfX_j(u)$ in two
    ways: there are $\sim \lambda_j$ such edges at each of the $n_i$ vertices
    in $\mfX_i(u)$, and (by Eq.~\eqref{eq:intersection}) there are $\sim \lambda_i$ such
    edges at each of the $n_j$ vertices in $\mfX_j(u)$. Hence, $n_i\lambda_j \sim n_j\lambda_i$.

    Now, using Eq.~\eqref{eq:lambda-ij}, 
    $\mu |\mcP'| \sim \mu n_i/\lambda_i \sim \mu n_j/\lambda_j =
    o(\lambda_j).$
    By the maximality of the cliques $C \in \mcP'$ in the subgraph of $G(\mfX)$ induced on $\mfX_J(u)$, for every
    distinct $C,C' \in \mcP'$ and $v \in C$, we have $|N(v)\cap C'| \le \mu$.
    Therefore, for $v \in C \cap \mfX_j(u)$, we have
    \begin{align*}
        \lambda_j - |\mfX_j(u)\cap C| &= |\mfX_j(u) \cap N(v)| - |\mfX_j(u)\cap C|\\
                                      &\le \left|\left(N(v)\cap
    \mfX_j(u)\right)\setminus C\right| \\ &\le \mu|\mcP'|  =  o(\lambda_j),
    \end{align*}
    so that $|\mfX_j(u)\cap C| \sim \lambda_j$, as desired.

    Now $\mcP'$ satisfies Definition~\ref{def:local-clique}, giving the
    desired contradiction.
\end{proof}

\begin{proof}[Proof of Lemma~\ref{lem:local-clique-strong}]
    Suppose for contradiction that no set $I$ of nondominant colors with $i \in
    I$ is such that $\mfX$ has strong $I$-local clique partitions. Without loss
    of generality, we may assume that $n_i$ is minimal for this property, i.e.,
    for every nondominant color $j$ with $n_j < n_i$, there is a set $J$ of
    nondominant colors with $j \in J$ such that $\mfX$ has strong $I$-local
    clique partitions.

    Let $I$ be the set of nondominant colors containing $i$ guaranteed by
    Lemma~\ref{lem:local-clique-exists}.

    Let $u \in V$ be such that some clique
    $C$ in the $I$-local clique partition at $u$ is not maximal in $G(\mfX)$.
    In particular, let $v \in V\setminus C$ be such that $C \subseteq N(v)$,
    and let $j = c(u,v)$. Then $j$ is a nondominant color, and $j \notin I$. 
    Furthermore, by the defining property of $I$ (the guarantee of
    Lemma~\ref{lem:local-clique-exists}), it is not the case that
    $n_i\sqrt{\mu/n_j} = o(\lambda_i)$. In particular we may take $n_j < n_i$,
    since otherwise, if $n_j \ge n_i$, then 
        $n_i\sqrt{\mu/n_j} \le \sqrt{n_i \mu} = o(\lambda_i)$
    by assumption. Now since $n_j < n_i$, also $\lambda_j = \Omega(\sqrt{n})$
    by assumption. Furthermore, by the minimality of $n_i$, there is a set $J$
    of nondominant colors with $j \in J$ such that $\mfX$ has strong $J$-local
    clique partitions on $J$. In particular, $i \notin J$.

    By the definition of $I$-local clique partitions,
    \[
        |N(v)\cap \mfX_i(u)| \ge |N(v)\cap \mfX_i(u)\cap C| \gtrsim
        \lambda_i.
    \]
    Now let $D$ be the clique containing $v$ in the $J$-local clique partition
    at $u$.  By the coherence of $\mfX$, for every $x \in \mfX_j(u)\cap D$, we
    have $|N(x)\cap \mfX_i(u)| \gtrsim \lambda_i$. Hence, there are $\gtrsim
    \lambda_j\lambda_i$ nondominant edges between $\mfX_j(u)\cap D$ and
    $\mfX_i(u)$. So, by the pigeonhole principle, some vertex $y \in \mfX_i(u)$
    satisfies
    \begin{align*}
        |N(y)\cap D \cap \mfX_j(u)| &\gtrsim \frac{\lambda_i\lambda_j}{n_i} =
        \omega\left(\sqrt{\frac{\mu}{n_i}} \lambda_j\right) \\
        &= \omega\left(\sqrt{\frac{\mu n}{n_i}}\right) = \omega(\mu).
    \end{align*}
    (The second inequality uses the assumption that $\sqrt{n_i\mu} =
    o(\lambda_i)$. The last inequality uses Proposition~\ref{prop:23}.) But
    then $D\setminus \{y\} \subseteq N(y)$ by
    Observation~\ref{obs:clique-basic}. Then $y \in D$ by the definition of a
    strong local clique partition, and so $i \in J$, a contradiction.

    We conclude that in fact $\mfX$ has strong $I$-local clique partitions.
\end{proof}


We finally complete the proof of Theorem~\ref{thm:clique-geometry}.

\begin{proof}[Proof of Theorem~\ref{thm:clique-geometry}]
    By Lemma~\ref{lem:local-clique-strong}, for every nondominant color $i$
    there is a set $I$ such that $\mfX$ has strong local clique partitions on
    $I$.  We claim that these sets $I$ partition the collection of nondominant
    colors.  Indeed, suppose that there are two sets $I$ and $J$ of nondominant
    colors such that $i \in I\cap J$ and $\mfX$ has strong $I$-local and
    $J$-local clique partitions. Let $u,v \in V$ be such that $c(u,v) =
    i$. By the uniqueness of the induced $\{i\}$-local clique partition at $u$
    (Lemma~\ref{lem:lcp-unique}), we have
    \[
    |K_I(u,v) \cap K_J(u,v)| \gtrsim \lambda_i = \omega(\mu),
    \]
    so $K_I(u,v) = K_J(u,v)$, and $I = J$.  In particular, for every
    nondominant color $i$, there exists a unique set $I$ of nondominant colors
    such that $\mfX$ has strong $I$-local clique partitions.
    
    We simplify our notation and write $K(u,v) = K_I(u,v)$ whenever $c(u,v) \in
    I$ and $\mfX$ has strong $I$-local clique partitions. By
    Lemma~\ref{lem:clique-eq-sym}, we have $K(u,v) = K(v,u)$ for all $u,v \in
    V$ with $c(u,v)$ nondominant.  Let $\mcG$ be the collection of cliques of
    the form $K(u,v)$ for $c(u,v)$ nondominant. Then $\mcG$ is an
    asymptotically uniform clique geometry.
\end{proof}

\subsection{Consequences of local clique partitions for the parameters
$\lambda_i$}

We conclude this section by analyzing some consequences for the parameters
$\lambda_i$ of our results on strong local clique partitions.

\begin{lemma}\label{lem:largelambda-impossible}
Let $\mfX$ be a PCC with $\rho = o(n^{2/3})$.  For every nondominant
color $i$, we have $\lambda_i < n_i - 1$.
\end{lemma}
\begin{proof}
    Suppose for contradiction that $\lambda_i = n_i-1$ for some nondominant
    color $i$.  For every nondominant color $j$, by Proposition~\ref{prop:23},
    we have $n_i\sqrt{\mu/n_j} = o(n_i) = o(\lambda_i)$. Furthermore, $n_i \mu
    = o(\lambda_i^2)$. Let $I$ be the set of nondominant colors with $i \in I$
    guaranteed by Lemma~\ref{lem:local-clique-exists}. In particular, $\mfX$
    has $I$-local clique partitions. In fact, since $\lambda_i = n_i-1$, for
    every vertex $u$ and every clique $C$ in the $I$-local clique partition at
    $u$, we have $C\cap \mfX_i(u) \sim n_i$. Hence, there is only one clique in
    the $I$-local clique partition at $u$, and so $\mfX_I(u)$ is a clique in
    $G(\mfX)$. For every vertex $u$, let $K_I(u) = \mfX_I(u)\cup \{u\}$. Then
    for every vertex $v\notin \mfX_I(u)$, we have $|N(v) \cap K_I(u)| \le \mu$.
    In particular, $K_I(u)$ is a maximal clique in $G(\mfX)$, and $\mfX$ has
    strong $I$-local clique partitions.
    
    Let $U,v \in V$ with $v \in \mfX_I(u)$, let $j = c(u,v) \in I$, and suppose
    $|K_I(v)\cap K_I(u)| > \mu$. Then $K_I(v) = K_I(u)$ by
    Lemma~\ref{lem:clique-overlap-eq}. Hence, by the coherence of $\mfX$,
    for any $w, x \in V$ with $x \in \mfX_j(w)$, $K_I(w) = K_I(x)$.  By
    applying this fact iteratively, we find that for any two vertices $y, z
    \in V$ such that there exists a path from $y$ to $z$ in $\mfX_j$, we
    have $z \in K_I(y)$, contradicting the primitivity of $\mfX$.  We conclude
    that $|K_I(v)\cap K_I(u)| \le \mu$ if $c(u,v) \in I$.  Hence, if we fix a
    vertex $u$ and count pairs of vertices $(v,w) \in \mfX_i(u)\times
    \mfX_I(u)$ with $c(w,v) = i$, we have
    \[
        n_i\sum_{j \in I}p^i_{ji} \le n_I\mu,
    \]
    where $n_I = \sum_{i \in I} n_i$.  In particular, for any vertex $u$ and $v
    \in \mfX_i(u)$, we have $|\mfX_{i^*}(v) \cap \mfX_I(u)| \le \mu n_I/n_i$.

    Fix a vertex $v \in V$. For some integer $\ell$, we fix
    distinct vertices $u_1,\ldots, u_{\ell}$ in $\mfX_{i^*}(v)$ such that for
    all $1 \le \alpha,\beta \le \ell$, we have $u_\alpha \notin
    \mfX_I(u_\beta)$. Since $|\mfX_{i^*}(v)\cap \mfX_I(u_\alpha)| \le \mu
    n_I/n_i$, we may take $\ell = \lfloor n_i/(2\mu) \rfloor$. As $\mu = o(n_i)$
    by Proposition~\ref{prop:23}, we therefore have $\ell = \Omega(n_i/\mu)$.
    
    By Lemma~\ref{lem:clique-overlap-eq}, for $\alpha \ne \beta$, we have
    $|\mfX_I(u_\alpha)\cap \mfX_I(u_\beta)| \le \mu$.  Hence, for any $1 \le
    \alpha \le \ell$, we have
    \[
        \left|\mfX_I(u_\alpha)\setminus \bigcup_{\beta \ne \alpha}
        \mfX_I(u_\beta)\right| \gtrsim
        n_I - \left\lfloor \frac{n_i}{2\mu}\right\rfloor\mu \ge \frac{n_I}{2}.
    \]
    But $c(u_\alpha,v) = i$, so $v \in K_I(u_\alpha)$, and so
    $\mfX_I(u_\alpha)\setminus\{v\} \subseteq K_I(u_\alpha)\subseteq\{v\} \subseteq N(v)$.
    Then
    \[
        |N(v)| \ge \left|\bigcup_{\alpha=1}^{\ell} \mfX_I(u_\alpha,
        v)\setminus\{v\}\right| \gtrsim
        \frac{n_I \ell}{2} = \Omega(\frac{n_i^2}{\mu}) = \omega(\rho)
    \]
    by Proposition~\ref{prop:23}. But this contradicts the definition of
    $\rho$. We conclude that $\lambda_i < n_i-1$.
\end{proof}

\begin{lemma}\label{lem:lambda-gap-new}
    Let $\mfX$ be a  PCC.  Suppose for some nondominant color $i$ we
    have $\lambda_i < n_i-1$. Then $\lambda_i \le (1/2)(n_i + \mu)$.
\end{lemma}
\begin{proof}
    Fix a vertex $u$, and suppose $\lambda_i < n_i-1$. Then there exist
    vertices $v,w \in \mfX_i(u)$ such that $c(v,w)$ is dominant. Then $|N(v)\cap
    N(w)| = \mu$. Therefore,
    \[
        2\lambda_i - \mu \le |(N(u)\cup N(v))\cap \mfX_i(u)| \le n_i.
    \]
\end{proof}

\begin{corollary}\label{cor:lambdai-not-ni}
    Suppose $\mfX$ is a PCC with $\rho = o(n^{2/3})$. Then for every
    nondominant color $i$, we have $\lambda_i \lesssim n_i/2$. 
\end{corollary}
\begin{proof}
    For every nondominant color $i$ we have $\lambda_i < n_i - 1$ by
    Lemma~\ref{lem:largelambda-impossible}. Then by
    Lemma~\ref{lem:lambda-gap-new} and Proposition~\ref{prop:23}, we have
    $\lambda_i \le (1/2)(n_i + \mu) \sim n_i/2$.
\end{proof}

\begin{corollary}\label{cor:lambdai-integer}
    Let $\mfX$ be a PCC with $\rho = o(n^{2/3})$ with an asymptotically uniform clique geometry $\mcC$.
    Then for every nondominant color $i$ there is an integer $m_i \ge 2$ such
    that $\lambda_i \sim n_i/m_i$.
\end{corollary}
\begin{proof}
    Fix a nondominant color $i$ and a vertex $u$, and let $m_i$ be the number
    of cliques $C \in \mcC$ such that $u \in C$ and $\mfX_i(u)\cap C \ne
    \emptyset$. So $n_i/m_i \sim \lambda_i$.  But by
    Corollary~\ref{cor:lambdai-not-ni}, we have $\lambda_i \lesssim n_i/2$, so
    $m_i \ge 2$.
\end{proof}


\section{Clique geometries in exceptional PCCs}\label{sec:2clique}

In this section we will classify PCCs $\mfX$ having a clique geometry
$\mcC$ and a vertex belonging to at most two cliques of $\mcC$. In particular,
we prove Theorem~\ref{thm:2clique-characterization}.

We will assume the hypotheses of Theorem~\ref{thm:2clique-characterization}.
So, $\mfX$ will be a PCC such that $\rho = o(n^{2/3})$, with an asymptotically
uniform clique geometry $\mcC$ and a vertex $u \in V$ belonging to at most two
cliques of $\mcC$.

\begin{lemma}\label{lem:2clique-all2}
    Under the hypotheses of Theorem~\ref{thm:2clique-characterization}, for $n$
    sufficiently large, every vertex $x \in V$ belongs to exactly two cliques
    of $\mcC$, each of order $\sim \rho/2$.
\end{lemma}
\begin{proof}
    Recall that by the definition of a clique geometry, for every vertex $x \in V$, every
    nondominant color $i$, and every clique $C$ in the geometry containing $x$,
    we have $|C \cap \mfX_i(x)| \lesssim \lambda_i$.  Thus, by
    Corollary~\ref{cor:lambdai-not-ni}, every vertex belongs to at least two
    cliques. In particular, $u$ belongs to exactly two cliques of $\mcC$, and
    (by Corollary~\ref{cor:lambdai-not-ni}) it follows that $\lambda_i \sim n_i/2$
    for every nondominant color $i$.  Hence, by the definition of a clique
    geometry, for every vertex $x$ and every nondominant color $i$, there are
    exactly two cliques $C \in \mcC$ such that $x \in C$ and $\mfX_i(x)\cap C
    \ne \emptyset$. 

    Let $i$ and $j$ be nondominant colors, and let $v \in \mfX_j(u)$, and let
    $C \in \mcC$ be the clique containing $u$ and $v$. Since $|\mfX_i(u)\cap
    C| \sim \lambda_i \sim n_i/2$, we have
    \begin{equation}\label{eq:twoclique}
        |N(v)\cap \mfX_i(u)|\gtrsim n_i/2.
    \end{equation}

    Now suppose for contradiction that some $x \in V$ belongs to at least three
    cliques of $\mcC$. Then there is some $C \in \mcC$ and nondominant color $i$ such
    that $x \in C$ but $\mfX_i(x)\cap C = \emptyset$. Let $j$ be a nondominant
    color such that $\mfX_j(x)\cap C \ne \emptyset$, and let $y \in
    \mfX_j(x)\cap C$.  By the coherence of $\mfX$ and Eq.~\eqref{eq:twoclique},
    we have $|N(y)\cap \mfX_i(x)| \gtrsim n_i/2$.

    But since there are exactly two cliques $C' \in \mcC$ such that $x \in
    C'$ and $\mfX_i(x)\cap C' \ne \emptyset$, then one of these cliques $C'$ is
    such that $|N(y)\cap \mfX_i(x) \cap C'| \gtrsim n_i/4$.  By
    Proposition~\ref{prop:23}, $n_i/4 = \omega(\mu)$ for $n$ sufficiently
    large. But then $C' \subseteq N(y)$, and $y \notin C'$, contradicting the
    maximality of $C'$.

    So every vertex $x \in V$ belongs to exactly two cliques of $\mcC$, and for
    each clique $C \in \mcC$ containing $x$ and each nondominant color $i$, we
    have $|\mfX_i(x)\cap C| \sim n_i/2$. It follows that $|C| \sim \rho/2$ for
    each $C \in \mcC$.
\end{proof}

\begin{lemma}\label{lem:2clique-rank}
Under the hypotheses of Theorem~\ref{thm:2clique-characterization}, for $n$
sufficiently large, $\mfX$ has rank at most four.
\end{lemma}
\begin{proof}
    Counting the number of vertex--clique incidences in $G(\mfX)$, we have $2n
    \sim |\mcC|(\rho/2 + 1) \sim |\mcC|\rho/2$ by Lemma~\ref{lem:2clique-all2}. On the other hand,
    every pair of distinct cliques $C,C' \in \mcC$ intersects in at most one
    vertex in $G(\mfX)$ by Property 2 of Definition~\ref{def:clique-geom}, and so $|\mcC|^2/2 \gtrsim n$.  It follows that $\rho
    \lesssim \sqrt{8n}$.  On the other hand, by Lemma~\ref{lem:diam2}, we
    have $n_i \gtrsim \sqrt{n}$ for every $i > 1$.  Since $\rho = \sum_{i >
    1}n_i$, for $n$ sufficiently large there are at most two nondominant
    colors.
\end{proof}

\begin{lemma}\label{lem:2clique-cross-edges}
    Under the hypotheses of Theorem~\ref{thm:2clique-characterization}, let $w$
    be a vertex, and $C_1, C_2 \in \mcC$ be the two cliques containing $w$.  Then
    for any $v \neq w$ in $C_1$, we have $|N(v) \cap (C_2\setminus\{w\})| \le 1$.
\end{lemma}
\begin{proof}
    We first note that by Lemma~\ref{lem:2clique-all2}, there are indeed
    exactly two cliques containing $w$. Note that $v \notin C_2$, since
    otherwise there are two cliques in $\mcC$ containing both $w$ and $v$.
    Suppose $v$ has two distinct neighbors $x,y$ in $C_2\setminus\{w\}$, so
    $x,y \notin C_1$ for the same reason.  Let $C_3 \in \mcC\setminus\{C_1\}$
    be the unique clique containing $v$ other than $C_1$. We have $x,y \in
    C_3$, but then $|C_2\cap C_3| \ge 2$, a contradiction.  So $v$ has at most
    one neighbor in $C_2\setminus\{w\}$. 
\end{proof}

The following result is folklore, although we could not find an explicit
statement in the literature.  A short elementary proof can be found inside the
proof of~\cite[Lemma 4.13]{cgss-linegraphs}.

\begin{lemma}\label{lem:srg}
    Let $G$ be a connected and co-connected strongly regular graph. If $G$ is
    the line-graph of a graph, then $G$ is isomorphic to $T(m)$,
    $L_2(m)$, or $C_5$.
\end{lemma}

\begin{proof}[Proof of Theorem~\ref{thm:2clique-characterization}]
    Let $H$ be the graph with vertex set $\mcC$, and an edge $\{C,C'\}$
    whenever $|C\cap C'|\ne 0$. Then $G(\mfX)$ is isomorphic to the line-graph
    $L(H)$.
    
    By Lemma~\ref{lem:2clique-rank}, $\mfX$ has rank at most four. By
    assumption (see Section~\ref{sec:overview}), $\mfX$ has rank at least
    three.
    
    Consider first the case that $\mfX$ has rank three.
    The nondiagonal colors $i,j$ of a rank three PCC $\mfX$ satisfy either $i^*
    = i$ and $j^* = j$, in which case $\mfX$ is a strongly regular graph, or
    $i^* = j$, in which case $\mfX$ is a ``strongly regular tournament,'' and
    $\rho = (n-1)/2$. We have assumed $\rho = o(n^{2/3})$, so $\mfX$ is a
    strongly regualr graph. But $G(\mfX)$ is the line-graph of the graph $H$,
    so by Lemma~\ref{lem:srg}, for $n > 5$, $\mfX$ is isomorphic to either
    $\mfX(T(m))$ or $\mfX(L_2(m))$.

    Suppose now that $\mfX$ has rank four, and let $I = \{2,3\}$ be the
    nondominant colors.  Fix $u \in V$, and let $C_1, C_2 \in \mcC$ be the
    cliques containing $u$ by Lemma~\ref{lem:2clique-all2}. By
    Corollary~\ref{cor:lambdai-not-ni} and Lemma~\ref{lem:2clique-cross-edges}, for any
    $i,j \in I$, not necessarily distinct, there exist $v \in C_1$ and $w \in
    C_2$ with $c(v,w)=1$, $c(v,u)=i$, and $c(u,w)=j$. Therefore, $p^1_{ij}\ge
    1$, and so $\mu \sum_{i,j > 1} p^1_{ij}\ge 4$.  Now let $x \in V$ be such that $c(u,x)=1$, and let
    $D_1,D_2 \in \mcC$ be the cliques containing $x$. For any $\alpha,\beta \in
    \{1,2\}$, we have $|C_\alpha \cap D_\beta| \le 1$, and so $\mu \le 4$. Hence, $\mu =
    4$, and $|C_\alpha \cap D_\beta| = 1$ for every $\alpha, \beta \in \{1,2\}$. Therefore, for any
    pair of distinct cliques $C,C' \in \mcC$ we have $|C\cap C'| = 1$, and so
    $H$ is isomorphic to $K_m$, where $m = |\mcC|$.

    In particular, every clique $C \in \mcC$ has order $m-1$, and so $n_2 + n_3
    = 2(m-2)$.
    
\begin{figure}
\begin{center}
\begin{tikzpicture}

\def\ptrad{0.06}

\definecolor{plum}{RGB}{142, 69, 33}

\coordinate (u) at (0,0);
\coordinate (v) at (0, -5);
\coordinate (w) at (-2,-1.7);
\coordinate (z) at (-2, -3.3);
\coordinate (y) at (2, -1.7);
\coordinate (x) at (2, -3.3);

\node[above] at (u) {$u$};
\node[below] at (v) {$v$};
\node[left] at (w) {$w$};
\node[right] at (x) {$x$};
\node[right] at (y) {$y$};
\node[left] at (z) {$z$};

\fill[color = black] (u) circle[radius=\ptrad];
\fill[color = black] (v) circle[radius=\ptrad];
\fill[color = black] (w) circle[radius=\ptrad];
\fill[color = black] (x) circle[radius=\ptrad];
\fill[color = black] (y) circle[radius=\ptrad];
\fill[color = black] (z) circle[radius=\ptrad];


\draw[dashed, thick](u)--(v);
\draw[dashed, thick](w)--(x);
\draw[dashed, thick](y)--(z);

\draw[-,  thick, color = red](u) -- (w);
\draw[-,  thick, color = red](u) -- (x);
\draw[-,  thick, color = blue](y) -- (u);
\draw[-,  thick, color = blue](z) -- (u);
\draw[-,  thick, color = red](v) -- (w);
\draw[-,  thick, color = blue](v) -- (x);
\draw[-,  thick, color = red](v) -- (y);
\draw[-,  thick, color = blue](v) -- (z);
\draw[-,  thick, color = blue](w) -- (y);
\draw[-,  thick, color = blue](w) -- (z);

\end{tikzpicture}
\end{center}
\caption{Two non-adjacent vertices $u$, $v$ and their common neighbors $w, x, y, z$.
The dashed line represents color 1. The red line represents color 2. The blue line represent color 3.}
\label{fig:twoclique}
\end{figure}
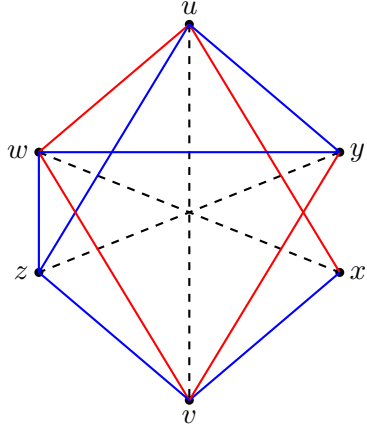

    Now we prove $2^* = 3$ and $3^* = 2$.
    Suppose for contradiction that colors $2$ and $3$ are symmetric. Fix
    two vertices $u$ and $v$ with $c(u,v)=1$. (See Figure~\ref{fig:twoclique}.) Then $N(u)\cap N(v) =
    \{w,x,y,z\}$ for some vertices $w,x,y,z \in V$, and there are four distinct
    cliques $C_1,C_2,C_3,C_4 \in \mcC$ such that every vertex in $A =
    \{u,v,w,x,y,z\}$ lies in the intersection of two of these cliques. Without
    loss of generality, assume $c(w,x)$ and $c(y,z)$ are dominant, and all other
    distinct pairs in $A$ except $(u,v)$ have nondominant color. Since for any $i,j \in I$
    we have $p^1_{ij}=1$, then without loss of generality, by considering the
    paths of length two from $u$ to $v$ in $G(\mfX)$, we have $c(u,w) =
    c(u,x)=2$, $c(u,y)=c(u,z)=3$, $c(v,w)=c(v,y)=2$, and $c(v,x)=c(v,z)=3$.
    Now $c(w,u)=c(w,v)=2$, and so $c(w,y)=c(w,z)=3$ since $p^1_{ij}=1$ for all
    $i,j \in I$ and $c(w,x)=1$. But now $c(u,z)=c(v,z)=c(w,z)=3$, which
    contradicts the fact that $p^1_{23} = p^1_{33}=1$ for $c(z, y) = 1$.   We
    conclude that $2^* = 3$ and $3^* = 2$.
\end{proof}

Finally, we prove that individualizing $O(\log n)$ vertices suffices to
completely split the PCCs of situation~(b) of Theorem~\ref{thm:2clique-characterization}.

\begin{proof}[Proof of Lemma~\ref{lem:2clique-split}]
    By Theorem~\ref{thm:2clique-characterization}, we may assume that $\mfX$ is a
    rank four PCC with a non-symmetric nondominant color $i$, and $G(\mfX)$ is
    isomorphic to $T(m)$ for $m = n_i+2$. (The other nondominant color is
    $i^*$.) In particular, every clique in $\mcC$ has order $n_i + 1$. We show
    that there is a set of size $O(\log n)$ which completely splits $\mfX$. 
    
    Note that $p^i_{ii^*} = p^i_{ii} = p^i_{i^*i}$ by
    Proposition~\ref{prop:cc}.  For any edge $\{u,v\}$ in $T(m)$, there are
    exactly $m-2=n_i$ vertices $w$ adjacent to both $u$ and $v$. Hence,
    considering all the possible of colorings of these edges in $\mfX$, we have
    \[
        n_i = p^i_{ii} + p^i_{ii^*} + p^i_{i^*i} + p^i_{i^*i^*} = 3p^i_{ii} +
        p^i_{i^*i^*}.
    \]
    Therefore, $p^i_{ii} + p^i_{i^*i^*} \ge n_i/3$, and 
    \begin{equation}\label{eq:piii-bound}
        p^i_{ii^*} + p^i_{i^*i} \le 2n_i/3.
    \end{equation}

    Fix an arbitrary clique $C \in \mcC$ and any pair of distinct vertices $u,v
    \in C$. (By possibly exchanging $u$ and $v$, we have $c(u, v) = i$.) Of the
    $n_i-1$ vertices $w$ in $C\setminus\{u,v\}$, at most $2n_i/3$ of these have
    $c(w,u) = c(w,v)$, by Eq.~\eqref{eq:piii-bound}. So, including $u$ and $v$
    themselves, there are at least $n_i/3 - 1 + 2 = n_i/3 + 1$ vertices $w \in
    C$ such that $c(w,u) \ne c(w,v)$.  Thus, if we individualize a random
    vertex $w \in C$, then $\Pr[c(w,u)\ne c(w,v)] > 1/3$. If this event occurs,
    then $u$ and $v$ get different colors in the stable refinement.
    Hence, if we
    individualize each vertex of $C$ independently at random with probability
    $6 \ln (n_i^2) / n_i$, then $u$ and $v$ get the same color in the stable
    refinement with probability $\le 1/n_i^4$. The union bound then gives a
    positive probability to every pair of vertices getting a different color,
    so there is a set $A$ of size $O(\log n_i)$ such that after individualizing
    each vertex in $A$ and refining to the stable coloring, every vertex in $C$
    has a uniqe color.  We repeat this process for another clique $C'$, giving
    every vertex in $C'$ a unique color at the cost of another $O(\log n_i)$
    individualizations. 

    On the other hand, every other clique $C'' \in \mcC$ intersects $C \cup C'$ in
    two uniquely determined vertices, since $G(\mfX)$ is isomorphic to
    $T(m)$. So, if $u \in C''$ and $v \notin C''$, then $u$ and $v$ get
    different colors in the stable refinement. Since every vertex lies in two
    uniquely determined cliques by Lemma~\ref{lem:2clique-all2}, it follows
    that every vertex gets a different color in the stable refinement.
\end{proof}


\section{Good triples}\label{sec:spielman}

In this section, we finally prove Lemma~\ref{lem:main-spielman-combined}.
For given nondominant colors $i$ and $j$, we will be interested in
quadruples of vertices $(u,w,x,y)$ with the following property:

\vspace{0.2cm}
\noindent \emph{Property $Q(i,j)$:} $c(x,y)=c(u,x)=c(u,y)=1$, $c(u,w)=i$, and
$c(w,x)=c(w,y)=j$ (See Figure~\ref{fig:spielman})

\begin{figure}
\begin{center}
\begin{tikzpicture}

\def\ptrad{0.06}

\coordinate (u) at (0,0);
\coordinate (w) at (0,-2);
\coordinate (x) at (-2, -3.5);
\coordinate (y) at (2, -3.5);

\node[right] at (u) {$u$};
\node[right] at (w) {$w$};
\node[below] at (x) {$x$};
\node[below] at (y) {$y$};

\fill[color = black] (u) circle[radius=\ptrad];
\fill[color = black] (w) circle[radius=\ptrad];
\fill[color = black] (x) circle[radius=\ptrad];
\fill[color = black] (y) circle[radius=\ptrad];

\draw[->,  thick](u) -- (0, -1.96);
\draw[->,  thick](w) -- (-1.96, -3.46);
\draw[->,  thick](w) -- (1.96, -3.46);

\draw[dotted, thick](u) -- (x);
\draw[dotted, thick](u) -- (y);
\draw[dotted, thick](x) -- (y);

\node[right] at (0, -1) {$i$};
\node[right] at (-.9, -2.75) {$j$};
\node[left] at (.9, -2.75) {$j$};

\end{tikzpicture}
\end{center}
\caption{$(u, w, x, y)$ has Property $Q(i, j)$. The dotted line represents the dominant color.}
\label{fig:spielman}
\end{figure}
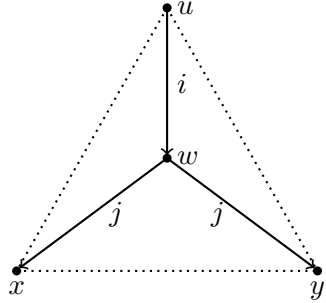

\begin{definition}[Good triple of vertices]\label{def:good-triple}
    For fixed nondominant colors $i,j$ and vertices $u,v$, we say a triple of
    vertices $(w,x,y)$ is \emph{good} for $u$ and $v$ if $(u,w,x,y)$
    has Property $Q(i,j)$, but there is no vertex $z$ such that
    $(v,z,x,y)$ has Property $Q(i,j)$.
\end{definition}

We observe that if $(w,x,y)$ is good for vertices $u$ and $v$, and both $x$ and $y$
are individualized, then $u$ and $v$ receive different colors after two
refinement steps. 

In the case of SRGs, there is only one choice of nondominant color, and
Property $Q(i,j)$ and Definition~\ref{def:good-triple} can be simplified: a
triple $(w,x,y)$ is good for $u$ and $v$ if $w,x,y,u$ induces a $K_{1,3}$, but
there is no vertez $z$ such that $z,x,y,v$ induces a $K_{1,3}$. Careful
counting of induced $K_{1,3}$ subgraphs formed a major part of Spielman's proof
of Theorem~\ref{thm:main-pcc} in the special case of SRGs~\cite{spielman-srg}.
Spielman's ideas inspired parts of this section. In particular, the proof of
the following lemma directly generalizes Lemmas~14 and~15
of~\cite{spielman-srg}.

\begin{lemma}\label{lem:many-good-triples-suffice}
    Let $\mfX$ be a PCC with $\rho = o(n^{2/3})$.  Suppose that for
    every distinct $u,v \in V$ there are nondominant colors $i$ and $j$ such
    that there are $\alpha = \Omega(n_in_j^2)$ good triples $(w,x,y)$ of vertices
    for $(u,v)$. Then there is a set of $O(n^{1/4}(\log n)^{1/2})$ vertices
    that completely splits $\mfX$.
\end{lemma}
\begin{proof}
    Let $S$ be a random set of vertices given by including each vertex in $
    V$ independently with probability $p$. Fix distinct $u,v \in V$. We
    estimate the probability that there is a good triple $(w,x,y)$ for $u$ and $v$
    such that $x,y \in S$.

    Let $T$ denote the set of good triples $(w,x,y)$ of vertices for $(u,v)$.
    Observe that any vertex $w \in \mfX_i(u)$ appears in at most $n_j^2$ good
    triples $(w,x,y)$ in $T$. On the other hand, if $w \in \mfX_i(u)$ is a
    random vertex, and $X$ is the number of pairs $x,y$ such that $(w,x,y) \in
    T$, then $\bbE[X] \ge \alpha/n_i$.  Therefore, 
    we have
    \[
        n_j^2 \Pr[X \ge \alpha/(2n_i)] + (1- \Pr[X \ge \alpha/(2n_i)])\alpha/(2n_i) \ge \bbE[X] \ge \alpha/n_i,
    \]
    and so, since $\alpha = \Omega(n_in_j^2)$ and $\alpha < n_in_j^2$ by definition,
    \[
        \Pr[X \ge \alpha/(2n_i)]  \ge \frac{1}{2n_j^2n_i/\alpha - 1} = \Omega(1).
    \]
    Let $U$ be the set of vertices $w \in \mfX_i(u)$ appearing in at least
    $\alpha/(2n_i)$ triples $(w,x,y)$ in $T$, so $|U| = \Omega(n_i)$. Now let $W
    \subseteq U$ be a random set given by including each vertex $w \in U$
    independently with probability $n/(3n_in_j)$.

    Fix a vertex $w \in W$ and a triple $(w,x,y) \in T$. Note that there are at
    most $p^1_{ij} \lesssim n_in_j/n$ vertices $w' \in U$ such that
    $c(w',x)=j$. Therefore, by the union bound, the probability that there is
    some $w' \ne w$ with $w' \in \mfX_{j^*}(x)\cap W$ is $\le 1/3$.  Similarly,
    the probability that there is some $w' \in \mfX_{j^*}(y)\cap W$ with $w'
    \ne w$ is at most $1/3$.  Hence, the probability that 
    \begin{equation}\label{eq:unique-nbr}
        \mfX_{j*}(x)\cap W = \mfX_{j^*}(y)\cap W = \{w\}
    \end{equation}
    is at least $1/3$.
    
    Now for any $w \in W$, let $T_w$ denote the set of pairs $x,y \in V$ such
    that $(w,x,y) \in T$ and Eq.~\eqref{eq:unique-nbr} holds. We have
    $\bbE[|T_w|] \ge \alpha / (6n_i) = \Omega(n_j^2)$. But in any case,
    $|T_w| \le n_j^2$. Therefore, for any $w \in W$, we have $|T_w| =
    \Omega(n_j^2)$ with probability $\Omega(1)$. Let $W' \subseteq W$ be the
    set of vertices $w$ with $|T_w| = \Omega(n_j^2)$. Since $\bbE[|W'|] =
    \Omega(|W|)$, we have $|W'| = \Omega(|W|)$ with probability $\Omega(1)$.
    Furthermore, $|W| = \Omega(n/n_j)$ with high probability by the Chernoff bound.
    
    Thus, there exists a set $W \subseteq \mfX_i(x)$ with a subset $W' \subseteq W$
    of size $\Omega(n/n_j)$ such that $|T_w| = \Omega(n_j^2)$ for every $w \in
    W'$.

    Now fix a  $w \in W'$. The probability that there are at least two vertices in
    $\mfX_j(w)\cap S$ is at least
    \begin{align*}
        1 - (1-p)^{n_j} - pn_j(1-p)^{n_j-1} &> 1 - e^{-pn_j} - pn_je^{-pn_j} 
        = \Omega(p^2n_j^2)
    \end{align*}
    if $pn_j < 1$, 
    using the Taylor expansion of the exponential function. Since $|T_w| = \Omega(n_j^2)$,
    the
    probability that there is a pair $(x,y) \in T_w$ with $x,y \in S$ is
    $\Omega(p^2n_j^2)$.

    Therefore, the probability that there is no $w \in W'$ with a pair $(x,y)
    \in T_w$ such that $x,y \in S$ is at most
    \[
        (1 - \Omega(p^2n_j^2))^{|W'|} \leq (1 - \Omega(p^2n_j^2))^{\eps n/n_j},
    \]
    for some constant $0 < \eps < 1$.
    For  $p = \beta \sqrt{\log n/(nn_j)}$ with a sufficiently large constant
    $\beta$, this probability is at most $1/(2n^2)$. Since $n_j \gtrsim
    \sqrt{n}$ for all $j$ by Lemma~\ref{lem:diam2}, we may take $p = \beta
    \sqrt{\log n / n^{3/2}}$ with a sufficiently large constant $\beta$.  Then,
    for any pair $u,v \in V$ of distinct vertices, the probability no good
    triple $(w,x,y)$ for $u$ and $v$ has $x,y \in S$ is at most $1/(2n^2)$. By
    the union bound, the probability that there is some pair $u,v \in V$ of
    distinct vertices such that no triple $(w,x,y)$ has the desired property is
    at most $1/2$. Therefore, after individualizing every vertex in $S$, every
    vertex in $V$ gets a unique color with probability at least $1/2$. By the
    Chernoff bound, we may furthermore assume that $|S| = O(n^{1/4}(\log
    n)^{1/2})$.
\end{proof}

We will prove that the hypotheses of Lemma~\ref{lem:many-good-triples-suffice}
hold separately for the case that $\lambda_k$ is small for some nondominant
color $k$ and the case that $\mfX$ has an asymptotically uniform clique geometry. Specifically, in
Section~\ref{sec:small-lambda} we prove the following lemma.

\begin{lemma}\label{lem:small-lambda-main}
    There is an absolute constant $\eps > 0$ such that the following holds. Let
    $\mfX$ be a  PCC with $\rho = o(n^{2/3})$ and a nondominant color
    $k$ such that $\lambda_k < \eps n^{1/2}$. Then there are two nondominant
    colors $i$ and $j$ such that for every pair of distinct vertices $u,v \in
    V$, there are $\Omega(n_i n_j^2)$ good triples of vertices for $u$ and $v$
    with respect to the colors $i$ and $j$.
\end{lemma}

Then, in Section~\ref{sec:large-lambda}, we prove the following lemma.

\begin{lemma}\label{lem:large-lambda-main}
    Let $\mfX$ be a  PCC with $\rho = o(n^{2/3})$ and a asymptotically uniform clique geometry
    $\mcC$ such that every vertex $u \in V$ belongs to at least three cliques
    in $\mcC$. Suppose that $n_i\mu = o(\lambda_i^2)$ for every nondominant
    color $i$. Then there are
    nondominant colors $i$ and $j$ such that for every pair of distinct vertices $u,v \in V$,  there are $\Omega(n_i n_j^2)$ good
    triples of vertices for $u$ and $v$ with respect to the colors $i$ and $j$.
\end{lemma}

Lemma~\ref{lem:main-spielman-combined} follows from
Lemmas~\ref{lem:many-good-triples-suffice},~\ref{lem:small-lambda-main}, and~Lemma~\ref{lem:large-lambda-main}
(noting for the latter that $\lambda_i = \Omega(n^{1/2})$ implies $n_i\mu = o(\lambda_i^2)$
by Proposition~\ref{prop:23}). \qed

Before proving Lemmas~\ref{lem:small-lambda-main}
and~\ref{lem:large-lambda-main}, we prove two smaller lemmas that will be
useful for both.

\begin{lemma}\label{lem:good-color-i-new}
    Let $\mfX$ be a  PCC with $\rho = o(n^{2/3})$.  Let $i$ be a
    nondominant color, and let $u,v\in V$ be distinct vertices.  We have  $|\mfX_i(u)\setminus N(v)|\gtrsim n_i/2$.  
\end{lemma}
\begin{proof}
    Let $j = c(u,v)$, and let $\eps = 2\mu/n_j$, so $\eps = o(1)$ by
    Proposition~\ref{prop:23}. By Corollary~\ref{cor:lambdai-not-ni}, we have
    $\lambda_i \lesssim n_i/2$. Therefore, by Lemma~\ref{lem:triple-count}, we
    have
    \[
        |\mfX_i(u)\cap N(v)| \le \max\left\{\frac{\lambda_i+1}{1-\eps},
        n_i\frac{\mu}{\eps n_j}\right\} \lesssim n_i/2.
    \]
\end{proof}

\begin{lemma}\label{lem:many-satisfy-star-new}
    Let $\mfX$ be a  PCC with $\rho = o(n^{2/3})$.  Let $i$ and $j$ be
    nondominant colors and let $u$ and $w$ be vertices with $c(u,w)=i$.
    Suppose that $|\mfX_j(w)\cap N(u)| \lesssim n_j/3$.
    Then there are $\gtrsim (1/9)n_j^2$ pairs of vertices $(x,y)$ with $c(x, y ) = 1$ such that
    $(u,w,x,y)$ has Property $Q(i,j)$.
\end{lemma}
\begin{proof}
    By Corollary~\ref{cor:lambdai-not-ni}, we have $\lambda_j \lesssim n_j/2$.
    Thus, for every vertex $x \in \mfX_j(w)\setminus N(u)$, there are at least
    $n_j - |\mfX_j(w)\cap N(u)| - \lambda_j \gtrsim n_j/6$ vertices $y \in
    \mfX_j(w) \setminus N(u)$ such that $(u,w,x,y)$ has Property $Q(i,j)$. 
    Since $|\mfX_j(w)\cap N(u)| \lesssim n_j/3$, the number of pairs $(x, y)$
    with $c(x, y) = 1$ such that $(u,w,x,y)$ has Property $Q(i,j)$ is 
    $\gtrsim (2n_j/3)(n_j/6) = (1/9)n_j^2$.
\end{proof}

\subsection{Good triples when some parameter $\lambda_k$ is small}\label{sec:small-lambda}

We prove Lemma~\ref{lem:small-lambda-main} in two parts, via the following two
lemmas.

\begin{lemma}\label{lem:small-lambda-main-new}
    There is an absolute constant $\eps > 0$ such that the following holds. Let
    $\mfX$ be a  PCC with $\rho = o(n^{2/3})$ and a nondominant color
    $k$ such that $\lambda_k < \eps n^{1/2}$ and $\lambda_{k^*} \lesssim n_k /
    3$. Then there are two nondominant colors $i$ and $j$ such that for every
    pair of distinct vertices $u,v \in V$, there are $\Omega(n_i n_j^2)$ good
    triples of vertices for $u$ and $v$ with respect to the colors $i$ and $j$.
\end{lemma}

\begin{lemma}\label{lem:clique-halfnk}
    Let $\tau$ be an arbitrary fixed positive integer. Let $\mfX$ be a  PCC with $\rho =
    o(n^{2/3})$ and a nondominant color $k$ such that $\lambda_k <
    n^{1/2}/(\tau +1)$ and $\lambda_{k^*} \gtrsim n_k/\tau$. Then for every pair of
    distinct vertices $u,v \in V$, there are $\Omega(n_k^3)$ good triples
    of vertices for $u$ and $v$ with respect to the colors $i = k^*$ and $j =
    k^*$.
\end{lemma}

We observe that Lemma~\ref{lem:small-lambda-main} follows from these two.

\begin{proof}[Proof of Lemma~\ref{lem:small-lambda-main}]
    Let $\eps'$ be the absolute constant given by Lemma~\ref{lem:small-lambda-main-new}, and
    let $\eps = \min\{\eps', 1/4\}$. Let $\mfX$ be a  PCC with $\rho =
    o(n^{2/3})$ and a nondominant color $k$ such that $\lambda_k < \eps
    n^{1/2}$. If $\lambda_{k^*} \gtrsim n_k/3$, then
    Lemma~\ref{lem:clique-halfnk} gives the desired result. Otherwise, the
    result follows from Lemma~\ref{lem:small-lambda-main-new}.
\end{proof}

We now turn our attention to proving Lemmas~\ref{lem:small-lambda-main-new}.

\begin{lemma}\label{lem:small-lambda-miss-v-new}
    Let $0 < \eps < 1$ be a constant, and $\mfX$ be a  PCC with $\rho = o(n^{2/3})$.
    Let $i$ and $k$ be nondominant colors, and let $w$ and $v$ be vertices such
    that $c(w,v)$ is dominant. Suppose $n_i \le n_\ell$ for all $\ell$, and
    $\lambda_{k} < \eps n^{1/2}$. Then there are $\lesssim \eps^2 n_k^2$ triples
    $(z,x,y)$ of vertices such that $x,y \in \mfX_{k^*}(w)$, $c(x, y)  = 1$ and $(v,z,x,y)$
    has Property $Q(i,k^*)$.
\end{lemma}
\begin{proof}
    First we observe that
    \[
        (p^1_{k^*k})^2 n_i \lesssim (n_k^2/n)^2n_i \le n_k^2 (\rho^3/n^2) =
        o(n_k^2).
    \]

    For every color $\ell$, there are exactly $p^1_{\ell i^*}$   vertices
    $z$ such that $c(v,z)=i$ and $c(w,z)=\ell$. For every such vertex
    $z$, there are at most $(p^\ell_{k^*k})^2$ pairs $x,y \in \mfX_{k^*}(w)$ with $c(x,y) = 1$ such that
    $(v,z,x,y)$ has Property $Q(i,k^*)$. Thus, by
    Proposition~\ref{prop:cc}, the total number of such triples is
    \begin{align*} 
        \sum_{\ell=1}^{r-1} p^1_{\ell i^*}(p^\ell_{k^*k})^2
        &\lesssim p^1_{1i^*}(p^{1}_{k^*k})^2 +
        \sum_{\ell=2}^{r-1}\frac{n_\ell n_i}{n}\left(\frac{n_k p^{k^*}_{\ell k^*}}{n_\ell}\right)^2 \\
        &\le o(n_k^2) +
        \sum_{\ell=2}^{r-1}\left(\frac{n_k^2}{n}\right)\left(p^{k}_{k\ell^*}\right)^2 \\
        &\le o(n_k^2) + \frac{n_{k}^2}{n}\lambda_{k}^2 \lesssim \eps^2 n_k^2.
    \end{align*}
\end{proof}

We finally complete the proof of Lemma~\ref{lem:small-lambda-main-new}.
\begin{proof}[Proof of Lemma~\ref{lem:small-lambda-main-new}]
    Let $\eps = \sqrt{1/18}$.     Let $i$ be a nondominant color minimizing $n_i$. By
    Lemma~\ref{lem:good-color-i-new}, we have  $|\mfX_i(u)\setminus
    N(v)|\gtrsim n_i/2$ for any pair of distinct vertices $u, v \in V$. 
    
    Let color $j = k^*$. Fix two distinct vertices $u$
    and $v$.
     Let $w \in \mfX_i(u)\setminus N(v)$. 
     Since $\lambda_j \lesssim n_j / 3$, we have $|\mfX_j(w)\cap N(u)| \lesssim n_j / 3$
        by Lemma~\ref{lem:triple-count} (with $\eps = \sqrt{9 \mu / n_i} = o(1)$).
    By Lemma~\ref{lem:many-satisfy-star-new}, there are $\gtrsim
    (1/9)n_j^2$ pairs of vertices $(x,y)$ with $c(x, y) = 1$ such that
    $(u,w,x,y)$ has Property $Q(i,j)$.  Furthermore, by
    Lemma~\ref{lem:small-lambda-miss-v-new}, for all but $\lesssim (1/18)
    n_j^2$ of these pairs $(x,y)$, the triple $(w,x,y)$ is good for $u$ and
    $v$. Since there are $\gtrsim n_i/2$ such vertices $w$, we have a total of
    $\gtrsim (1/36)n_in_j^2$ triples $(w,x,y)$ that are good for $u$ and $v$.
\end{proof}

Now we prove Lemma~\ref{lem:clique-halfnk}.

\begin{lemma}\label{lem:large-lambda-miss-v-new}
    Let $\mfX$ be a PCC with $\rho = o(n^{2/3})$ and strong $I$-local
    clique partitions for some set $I$ of nondominant colors. Let $j \in
    I$ be a color  such that $\lambda_j = \Omega(n_j)$. Let $w$ and $v$ be vertices such
    that $c(w,v)=1$. Then for any nondominant color $i$ with $n_i \le n_j$,
    there are $o(n_j^2)$ triples $(z,x,y)$ of vertices such that $x,y \in
    \mfX_j(w)$, $c(x, y) =1$ and $(v,z,x,y)$ has Property $Q(i,j)$.
\end{lemma}
\begin{proof}
    Fix a nondominant color $i$, and let $T$ be the set of triples $(z,x,y)$
    such that $x,y \in \mfX_j(w)$, $c(x, y) = 1$ and $(v,z,x,y)$ has Property $Q(i,j)$.
    
    If $c(z,w)=1$, then $|\mfX_j(z)\cap \mfX_j(w)| = p^1_{jj^*}$, and so there are at most
    $(p^1_{jj^*})^2$ pairs $x,y \in \mfX_j(w)$ with $c(x, y) = 1$ such that
    $(z,x,y) \in T$. Then since $c(v,z) = i$ whenever $(z,x,y) \in T$, the
    total number of triples $(z,x,y) \in T$ such that $C(z,w)=1$ is 
    at most
    \[
        (p^1_{jj^*})^2 n_i \lesssim (n_j^2/n)^2n_i \le n_j^2 (\rho^3/n^2) =
        o(n_j^2),
    \]
    where the first inequality follows from Proposition~\ref{prop:cc}~(iii)
    and~(iv), and the relation $n_1 \sim n$.

    Since $c(w,v)=1$, there are $\le \rho n_i/n$ vertices $z \in \mfX_i(v)\cap N(w)$.
    Suppose $z \in \mfX_i(v) \cap N(w) $. Let $\mcC$ denote the collection of cliques
    partitioning $\mfX_I(w)$. If some clique in $\mcC$ contains $z$, let $C$ be
    that clique; otherwise, let $C = \emptyset$. Since $\mcC$ partitions
    $\mfX_j(w)$ into $\sim n_j/\lambda_j = O(1)$ cliques for each $w \in V$,
    and $|\mfX_j(z)\cap C'| \le \mu$ for every clique $C' \in \mcC$ with $C \ne C'$, we
    therefore have $|(\mfX_j(w)\cap \mfX_j(z))\setminus C| \lesssim \mu n_j/\lambda_j= O(\mu)$.  But then
    there are at most
    \[
        |\mfX_j(w)\cap \mfX_j(z)|\cdot|(\mfX_j(w)\cap\mfX_j(z))\setminus C|
        = O(n_j\mu)
    \]
    pairs $x,y \in \mfX_j(w)$ with $c(x, y) = 1$ such that $(v,z,x,y) \in T$, for a total of at
    most $O(n_j \mu (\rho n_i)/n) = o(n_j^2)$ triples $(z,x,y) \in T$ with
    $c(z,w) \ne 1$.
\end{proof}

\begin{proof}[Proof of Lemma~\ref{lem:clique-halfnk}]
    For every nondominant color $\ell$, by Proposition~\ref{prop:23}, we have
    $n_k\sqrt{\mu/n_\ell} = o(n_k) = o(\lambda_{k^*})$. Similarly, $n_k \mu =
    o(\lambda_{k^*}^2)$. Hence, by Lemma~\ref{lem:local-clique-exists} and
    Definition~\ref{def:str-local-clique}, there is a set $I$ of nondominant
    colors with $k^* \in I$ such that $\mfX$ has strong $I$-local clique
    partitions.  Since $\lambda_k < n^{1/2} / (\tau + 1) \lesssim n_k / (\tau +
    1)$ by Lemma~\ref{lem:diam2}, and since $\lambda_{k^*} \gtrsim n_k / \tau$,
    the by the definition of an $I$-local clique partition, $k \notin I$.
    Hence, by the definition of a strong $I$-local clique partition and
    Observation~\ref{obs:clique-basic}, for a vertex $x$ and a vertex $y \in
    \mfX_k(x)$, $|N(y) \cap \mfX_{k^*}(x)| \le \tau \mu = o(n_k)$.
    
    On the other hand, by Corollary~\ref{cor:lambdai-not-ni}, we  have
    $\lambda_{k^*} \lesssim n_k/2$, and hence $\lambda_k \le n^{1/2}/3 \lesssim
    n_k/3$ by Lemma~\ref{lem:diam2}.

    Fix $u,v \in V$. By Lemma~\ref{lem:good-color-i-new}, we have
    $|\mfX_{k^*}(u)\setminus N(v)| \gtrsim n_{k^*}/2$. Let $w \in \mfX_{k^*}(u)\setminus
    N(v)$. We have $c(w,u) = k$ and so $|\mfX_{k^*}(w)\cap N(u)| = o(n_k)$. By
    Lemma~\ref{lem:many-satisfy-star-new}, there are $\Omega(n_k^2)$ pairs of non-adjacent
    vertices $(x,y)$ such that $(u,w,x,y)$ has Property $Q(k^*,k^*)$.
    But by Lemma~\ref{lem:large-lambda-miss-v-new}, there are $o(n_k^2)$
    triples $(x,y,z)$ of vertices such that $x,y \in \mfX_{k^*}(w)$, $c(x, y) = 1$ and
    $(v,z,x,y)$ has Property $Q(k^*,k^*)$. So there are $\Omega(n_k^2)$
    pairs $(x,y)$ of vertices such that $(w,x,y)$ is good for $u$ and $v$ with
    respect to colors $i = k^*$ and $j = k^*$. Since we have $\Omega(n_k)$
    choices for vertex $w$, there are in total $\Omega(n_k^3)$ good triples, as
    desired.
\end{proof}

\subsection{Good triples with clique geometries}\label{sec:large-lambda}

We now prove Lemma~\ref{lem:large-lambda-main}.

\begin{lemma}\label{lem:good-color-j}
    Let $\mfX$ be a  PCC with $\rho = o(n^{2/3})$ and an asymptotically uniform clique geometry
    $\mcC$ such that every vertex $u \in V$ belongs to at least three cliques
    in $\mcC$. Suppose that $n_i\mu = o(\lambda_i^2)$ for every nondominant
    color $i$. 
     Then, for any nondominant color $i$, there is a nondominant color $j$ such that 
    for every $u, w$ with $w \in \mfX_i(u)$, we have
    $|\mfX_j(w)\cap N(u)| \lesssim n_j/3$. 
\end{lemma}
\begin{proof}
    Let $C \in \mcC$ be the unique clique such that $u,w \in C$.
    
    If $\lambda_{i^*} \lesssim n_i/3$, we let $j = i^*$. Then by the maximality
    of the cliques in $\mcC$ partitioning $\mfX_j(w)$, we have \[|(\mfX_j(w)\cap
    N(u))\setminus C| \le \mu n_j/\lambda_j = o(\lambda_j)\] by Observation \ref{obs:clique-basic}. So $|\mfX_j(w)\cap
    N(u)| \lesssim n_j/3$.

    Otherwise, by Corollary~\ref{cor:lambdai-integer}, $\lambda_{i^*} \sim n_i/2$, and so there is at most one
    clique $C' \in \mcC$ with $C' \ne C$ such that $|\mfX_{i^*}(w) \cap C'| \ne
    0$. Therefore, there is a clique $C''$ such that $|\mfX_{i^*}(w) \cap C''|
    = 0$. Let $j$ be a nondominant color such that $|\mfX_j(w)\cap C''| \sim
    \lambda_j$. Again, by the maximality of the cliques in $\mcC$ partitioning
    $\mfX_j(w)$ and Observation \ref{obs:clique-basic}, we have \[|\mfX_j(w)\cap N(u)| \lesssim \mu n_j/\lambda_j =
    o(\lambda_j),\] as desired. Furthermore, this inequality does not depend on
    the choice of $w$ by the coherence of $\mfX$.
\end{proof}

\begin{lemma}\label{lem:large-lambda-miss-v}
    Let $\mfX$ be a  PCC with $\rho = o(n^{2/3})$ and an asymptotically uniform clique geometry
    $\mcC$, let $w$ and $v$ be vertices such that $c(w,v)$ is dominant, and let
    $j$ be a nondominant color such that $\mu = o(\min\{\lambda_j,
    \lambda_j^4/n_j^2\})$. Then for any nondominant color $i$, there are
    $o(n_j^2)$ triples $(z,x,y)$ of vertices such that $x,y \in \mfX_j(w)$, $c(x, y) = 1$ and
    $(v,z,x,y)$ has Property $Q(i,j)$.
\end{lemma}
\begin{proof}
    Fix a nondominant color $i$, and let $T$ be the set of triples $(z,x,y)$
    such that $x,y \in \mfX_j(w)$, $c(x, y) = 1$ and $(v,z,x,y)$ has Property $Q(i,j)$.
    
    If $c(z,w)=1$, then $|\mfX_j(z)\cap \mfX_j(w)| = p^1_{jj^*}$, and so there are at most
    $(p^1_{jj^*})^2$ pairs $x,y \in \mfX_j(w)$ such that $(z,x,y) \in T$, for a total of
    at most
    \[
        (p^1_{jj^*})^2 n_i \lesssim (n_j^2/n)^2n_i \le n_j^2 (\rho^3/n^2) =
        o(n_j^2).
    \]
    triples $(z,x,y) \in T$ with $c(z,w) =1$.

    Since $c(w,v)=1$, there are $\le \mu$ vertices $z \in \mfX_i(v)\cap N(w)$.
    Suppose $z \in \mfX_i(v) \cap  N(w)$. Let $C$ be the clique in $\mcC$ containing both $z$ and $w$.
    Note that $|C\cap
    \mfX_j(w)|\lesssim \lambda_j$. For any $C_w,C_z \in \mcC$ with $w \in C_w$
    and $z \in C_z$ such that $C_w \ne C$ and  $C_z \ne C$, we have $|C_w\cap C_z| \le
    1$.  Since $\mcC$ partitions $\mfX_j(u)$ into $\sim n_j/\lambda_j$ cliques
    for each $u \in V$, we therefore have $|(\mfX_j(w)\cap \mfX_j(z)) \setminus C|
    \lesssim (n_j/\lambda_j)^2$.  But then there are at most
    \begin{align*}
        |\mfX_j(w)\cap \mfX_j(z)|\cdot|(\mfX_j(w)\cap\mfX_j(z))\setminus C|
        &\lesssim (\lambda_j + (n_j/\lambda_j)^2)(n_j/\lambda_j)^2 \\
        &\lesssim n_j^2/\lambda_j + (n_j/\lambda_j)^4 = o(n_j^2/\mu)
    \end{align*}
    pairs $x,y \in \mfX_j(w)$ with $c(x, y) = 1$ such that $(v,z,x,y) \in T$, for a total of at
    most $o(n_j^2)$ triples $(z,x,y) \in T$ with $c(z,w) \ne 1$.
\end{proof}

\begin{proof}[Proof of Lemma~\ref{lem:large-lambda-main}]
    Let $u$ and $v$ be two distinct vertices. By Lemma~\ref{lem:good-color-i-new},
    there is a nondominant color $i$ such that $|\mfX_i(u)\setminus N(v)|
    \gtrsim n_i/2$.  By Lemma~\ref{lem:good-color-j}, there is a nondominant
    color $j$ such that for every $w \in \mfX_i(u)$, we have $|\mfX_j(w)\cap
    N(u)| \lesssim n_j/3$. 
    
    Let $w \in \mfX_i(u)\setminus N(v)$. 
    By Lemma~\ref{lem:many-satisfy-star-new},
    there are $\Omega(n_j^2)$ pairs of vertices $(x,y)$ with $c(x, y) = 1$ such that $(u,w,x,y)$
    has Property $Q(i,j)$.  Furthermore, 
    since $\mu$ is a positive integer and $\mu = o(\lambda_j^2 / n_j)$, we have $\mu = o(\min\{\lambda_j, \lambda_j^4 / n_j^2\})$.
    By
    Lemma~\ref{lem:large-lambda-miss-v}, for all but $o(n_j^2)$ of these pairs
    $(x,y)$, the triple $(w,x,y)$ is good for $u$ and $v$. Since there are
    $\gtrsim n_i/2$ such vertices $w$, we have a total of $\Omega(n_in_j^2)$
    triples $(w,x,y)$ that are good for $u$ and $v$.
\end{proof}



\section{Conclusion}\label{sec:conclusion}

We have proved that except for the readily identified exceptions of
complete, triangular, and lattice graphs, a PCC is completely split after
individualizing $\wtO(n^{1/3})$ vertices and applying naive color refinement.
Hence, with only those three classes of exceptions, PCCs have at most
$\exp(\wtO(n^{1/3}))$ automorphisms. As a corollary, we have given a CFSG-free classifcation of
the primitive permutation groups of sufficiently large degree $n$ and order not
less than $\exp(\wtO(n^{1/3}))$.

As we remarked in the
introduction, Theorem~\ref{thm:main-pcc} is tight up to polylogarithmic
factors in the exponent, as evidenced by the Johnson and Hamming schemes. 
However, further progress may be possible for Babai's conjectured classification of PCCs with large
automorphism groups, Conjecture~\ref{conj:babai}.

The PCCs with large automorphism groups
appearing in Conjecture~\ref{conj:babai} are all in fact association schemes,
i.e., they satisfy $i^* = i$ for every color $i$. Intuitively, the presence of
asymmetric colors (oriented constituent graphs) should reduce the number of
automorphisms. On the other hand, the possibility of asymmetric colors greatly
complicates our analysis. For example, situation~(2) of
Theorem~\ref{thm:2clique-characterization} and Lemma~\ref{lem:2clique-split}
could be eliminated, and the proof of Lemma~\ref{lem:diam2} would become
straightforward, for association schemes.  Hence, a reduction to the case of
association schemes would be desirable.

\begin{question}
    Is it the case that every sufficiently large PCC with at least
    $\exp(n^{\eps})$ automorphisms is an association scheme?
\end{question}

The best that is known in this direction is the result of the present paper: if $\mfX$ is a PCC
that is not an association scheme, then $|\Aut(\mfX)| \le \exp(\wtO(n^{1/3}))$.

We comment on the bottlenecks for the current analysis. Below the threshold
$\rho = o(n^{1/3})$, we in fact have the improved bound $|\Aut(\mfX)| \le
\exp(\wtO(n^{1/4}))$ when $\mfX$ is nonexceptional, by
Lemma~\ref{lem:main-spielman-combined}. This region of the parameters is
therefore not a bottleneck for improving the current analysis. On the other
hand, Conjecture~\ref{conj:babai} suggests that nonexceptional PCCs $\mfX$ with
$\rho = o(n^{1/3})$ should satisfy $|\Aut(\mfX)| \le \exp(O(n^{o(1)}))$.

When $\rho = \Theta(n^{2/3})$, the Johnson scheme $J(m,3)$ and $H(3,m)$ emerge
as additional exceptions, with automorphism groups of order
$\exp(\Theta(n^{1/3}\log n))$.  The bottleneck for the current analysis is
above this threshold. In this region of the parameters, we analyze the
distinguishing number $D(i)$ of the edge-colors $i$. When $\rho < n/3$, our
best bounds are $D(1) = \Omega(\rho)$ and $D(i) \ge \Omega(D(1)n/\rho)^{2/3}$
from Lemma~\ref{lem:Di-technical}. When $\rho \ge n/3$, we use the estimate
$D(i) \ge \Omega(\sqrt{\rho n_i/\log n})$ from
Lemma~\ref{lem:nondominant-main}. Neither bound simplifies to anything better
than $D(i) = \Omega(n^{2/3})$ in any portion of the range range $\rho =
\Omega(n^{2/3})$.

Babai makes the following conjecture~\cite[Conjecture 7.4]{babai-cc}, which
would give an improvement when $\rho \ge n^{2/3 + \eps}$.

\begin{conjecture}
Let $\mfX$ be a PCC. Then there is a set of $O((n/\rho)\log n)$ vertices which
completely splits $\mfX$ under naive refinement. In particular, $\Aut(\mfX) \le
\exp((n/\rho)\log^2 n)$.
\end{conjecture}

Again, the best known bound is $|\Aut(\mfX)| \le \exp(\wtO(n^{1/3}))$, from the
present paper, although the conjecture has been confirmed for PCCs of bounded
rank~\cite{babai-cc}.

\bibliographystyle{abbrv}
\bibliography{primcc}



\end{document}